\documentclass{amsart}
\usepackage{amsmath, amsfonts, amssymb}
\usepackage{stmaryrd}
\usepackage{amscd}
\usepackage{hhline}
\usepackage{graphicx}
\usepackage{tableau}
\usepackage[colorlinks=true, linkcolor=blue,
  citecolor=blue, urlcolor=blue]{hyperref}

\def\newthm#1#2{\newtheorem{#1}[dummy]{#2}%
  \expandafter\def\csname#2\endcsname##1{\hyperref[#1:##1]{#2~\ref*{#1:##1}}}}
\newthm{thm}{Theorem}
\newthm{lemma}{Lemma}
\newthm{prop}{Proposition}
\newthm{cor}{Corollary}
\newthm{conj}{Conjecture}
\newthm{cons}{Consequence}
\newthm{fact}{Fact}
\newthm{obs}{Observation}
\newthm{guess}{Guess}
\theoremstyle{definition}
\newthm{defn}{Definition}
\newthm{remark}{Remark}
\newthm{example}{Example}
\newthm{question}{Question}
\newthm{notation}{Notation}

\newcommand{\Section}[1]{\hyperref[sec:#1]{Section~\ref*{sec:#1}}}
\newcommand{\Table}[1]{\hyperref[tab:#1]{Table~\ref*{tab:#1}}}
\newcommand{\eqn}[1]{\hyperref[eqn:#1]{(\ref*{eqn:#1})}}

\DeclareMathOperator{\GL}{GL}
\DeclareMathOperator{\SL}{SL}
\DeclareMathOperator{\Sp}{Sp}
\DeclareMathOperator{\SO}{SO}
\DeclareMathOperator{\Gr}{Gr}

\DeclareMathOperator{\LG}{LG}

\DeclareMathOperator{\SF}{SF}
\DeclareMathOperator{\OG}{OG}

\DeclareMathOperator{\Span}{Span}

\DeclareMathOperator{\QH}{QH}
\DeclareMathOperator{\QK}{QK}

\DeclareMathOperator{\codim}{codim}
\DeclareMathOperator{\Lie}{Lie}

\DeclareMathOperator{\diag}{diag}

\newcommand{\ssm}{\smallsetminus}
\newcommand{\bP}{{\mathbb P}}

\newcommand{\C}{{\mathbb C}}

\newcommand{\Q}{{\mathbb Q}}
\newcommand{\Z}{{\mathbb Z}}
\newcommand{\N}{{\mathbb N}}

\newcommand{\cO}{{\mathcal O}}
\newcommand{\cP}{{\mathcal P}}

\newcommand{\euler}[1]{\chi_{_{#1}}}
\newcommand{\pt}{\text{point}}

\newcommand{\al}{{\alpha}}
\newcommand{\be}{{\beta}}
\newcommand{\ga}{{\gamma}}
\newcommand{\vep}{{\varepsilon}}
\newcommand{\la}{{\lambda}}
\newcommand{\La}{{\Lambda}}

\newcommand{\ev}{\operatorname{ev}}

\newcommand{\wh}{\widehat}
\newcommand{\wb}{\overline}
\newcommand{\ov}{\overline}
\newcommand{\pic}[2]{\includegraphics[scale=#1]{#2}}
\newcommand{\ignore}[1]{}

\newcommand{\Mb}{\wb{\mathcal M}}
\newcommand{\ch}{\operatorname{ch}}
\newcommand{\Gq}{\Gamma\llbracket q\rrbracket}
\newcommand{\noin}{\noindent}

\begin{document}

\title{A Chevalley formula for the equivariant quantum $K$-theory of
  cominuscule varieties}

\date{June 6, 2017}

\author{Anders~S.~Buch}
\address{Department of Mathematics, Rutgers University, 110
  Frelinghuysen Road, Piscataway, NJ 08854, USA}
\email{asbuch@math.rutgers.edu}

\author{Pierre--Emmanuel Chaput}
\address{Domaine Scientifique Victor Grignard, 239, Boulevard des
  Aiguillettes, Universit{\'e} de Lorraine, B.P.  70239,
  F-54506 Vandoeuvre-l{\`e}s-Nancy Cedex, France}
\email{pierre-emmanuel.chaput@univ-lorraine.fr}

\author{Leonardo~C.~Mihalcea}
\address{Department of Mathematics, Virginia Tech University, 460
  McBryde, Blacksburg VA 24060, USA}
\email{lmihalce@math.vt.edu}

\author{Nicolas Perrin}
\address{Laboratoire de Math\'ematiques de Versailles, UVSQ, CNRS,
  Universit\'e Paris-Saclay, 78035 Versailles, France}
\email{nicolas.perrin@uvsq.fr}



\thanks{The first author was supported in part by NSF grants
  DMS-1205351 and DMS-1503662.}

\thanks{The third author was supported in part by NSA grants
  H98230-13-1-0208 and H98320-16-1-0013, and a Simons Collaboration
  Grant.}

\thanks{The fourth author was supported by a public grant as part of
  the Investissement d'avenir project, reference ANR-11-LABX-0056-LMH,
  LabEx LMH}

\begin{abstract}
  We prove a type-uniform Chevalley formula for multiplication with
  divisor classes in the equivariant quantum $K$-theory ring of any
  cominuscule flag variety $G/P$.  We also prove that multiplication
  with divisor classes determines the equivariant quantum $K$-theory
  of arbitrary flag varieties.  These results prove a conjecture of
  Gorbounov and Korff concerning the equivariant quantum $K$-theory of
  Grassmannians of Lie type A.
\end{abstract}

\maketitle

\markboth{A.~BUCH, P.--E.~CHAPUT, L.~MIHALCEA, AND N.~PERRIN}
{A CHEVALLEY FORMULA FOR EQUIVARIANT QUANTUM $K$-THEORY}

\section{Introduction}

Let $X = G/P$ be a flag variety defined by a semisimple complex Lie
group $G$ and a parabolic subgroup $P$.  The (small) {\em equivariant
  quantum $K$-theory ring\/} $\QK_T(X)$ of Givental
\cite{givental:wdvv} is a common generalization of the main cohomology
theories considered in Schubert calculus, including $K$-theory,
equivariant cohomology, and quantum cohomology.  Equivariant quantum
$K$-theory is the most general theory for which the associated
Schubert structure constants have positivity properties that are
either known \cite{graham:positivity, mihalcea:positivity,
  brion:positivity, anderson.griffeth.ea:positivity,
  anderson.chen:positivity} or conjectured \cite{lenart.maeno:quantum,
  lenart.postnikov:affine, buch.mihalcea:quantum}.  In this paper we
prove a Chevalley formula that combinatorially determines the ring
$\QK_T(X)$ when $X$ is a {\em cominuscule variety}, that is, a
Grassmann variety $\Gr(m,n)$ of type A, a Lagrangian Grassmannian
$\LG(n,2n)$, a maximal orthogonal Grassmannian $\OG(n,2n)$, a quadric
hypersurface $Q^n$, or one of two exceptional varieties called the
Cayley plane $E_6/P_6$ and the Freudenthal variety $E_7/P_7$.

The Schubert structure constants of the ring $\QK_T(X)$ are defined as
polynomial expressions in the (equivariant, $K$-theoretic) {\em
  Gromov-Witten invariants\/} of $X$.  The non-equivariant
Gromov-Witten invariants used to define the quantum $K$-theory ring
$\QK(X)$ can be computed for some spaces using the reconstruction
theorem of Lee and Pandharipande
\cite{lee.pandharipande:reconstruction} together with the
$K$-theoretic $J$-function of Taipale \cite{taipale:k-theoretic}.
However, infinitely many Gromov-Witten invariants are required to
evaluate the product of two Schubert classes in $\QK(X)$, and as a
result we do not know much about the structure of this ring for
general flag varieties.  Iritani, Milanov, and Tonita have recently
used an enhanced reconstruction theorem to compute the quantum
$K$-theory of the variety $\SL(3)/B$ of complete flags in $\C^3$
\cite{iritani.milanov.ea:reconstruction}.

When $X$ is a Grassmann variety of type A, the non-equivariant ring
$\QK(X)$ is determined by a Pieri formula of Buch and Mihalcea for
products that involve special Schubert classes
\cite{buch.mihalcea:quantum}.  More generally, when the Picard group
of $X$ has rank one, it has been proved by the authors of the present
paper that all products of Schubert classes in the equivariant ring
$\QK_T(X)$ contain only finitely many non-zero terms
\cite{buch.chaput.ea:finiteness, buch.chaput.ea:rational}.  As a
consequence, the multiplication table of $\QK_T(X)$ can be determined
from finitely many Gromov-Witten invariants.

Our main result is a type-uniform Chevalley formula that describes the
product of an arbitrary Schubert class with the divisor class in
$\QK_T(X)$ when $X$ is any cominuscule variety.  The Schubert classes
on a cominuscule variety can be indexed by diagrams of boxes that
generalize the Young diagrams known from the Schubert calculus of
classical Grassmannians.  Our Chevalley formula can be formulated in
terms of adding or removing boxes from these diagrams, and can
equivalently be formulated in terms of operations on Weyl group
elements.  Chevalley formulas have previously been proved for the
equivariant $K$-theory ring $K^T(X)$ by Lenart and Postnikov
\cite{lenart.postnikov:affine}, and for the equivariant quantum
cohomology ring $\QH_T(X)$ by Mihalcea \cite{mihalcea:equivariant*1};
these formulas are valid for arbitrary flag varieties.

The Gromov-Witten invariants of any cominuscule variety $X$ can be
computed using the `{\em quantum equals classical\/}' theorem
\cite{buch.kresch.ea:gromov-witten, chaput.manivel.ea:quantum*1,
  buch.mihalcea:quantum, chaput.perrin:rationality}, which states that
any (3 point, genus zero) Gromov-Witten invariant of $X$ is equal to a
$K$-theoretic triple intersection on a related space.  When $X$ is a
Grassmannian of type A, this related space is a two-step partial flag
variety of kernel-span pairs \cite{buch:quantum}.  Our Chevalley
formula is proved using an alternative version of the quantum equals
classical theorem, which expresses the Gromov-Witten invariants of $X$
in terms of {\em projected Gromov-Witten varieties\/}
\cite{knutson.lam.ea:positroid, buch.chaput.ea:projected}.  Here a
projected Gromov-Witten variety is the closure of the union of all
rational curves of a fixed degree that pass through two opposite
Schubert varieties in $X$.  Our proof shows that the quantum terms in
the Chevalley formula can be obtained by applying a linear operator to
the classical terms.  We also provide an alternative geometric proof
of Lenart and Postnikov's Chevalley formula
\cite{lenart.postnikov:affine}, specialized to the equivariant
$K$-theory of cominuscule varieties.  This proof is type-uniform for
all minuscule varieties, but requires specialized arguments for
Lagrangian Grassmannians and quadrics of odd dimension.  It would be
interesting to investigate whether similar geometric arguments extend
beyond the case of cominuscule flag varieties.

Gorbounov and Korff have recently used ideas from integrable systems
to define an algebra $qh^*(\Gr(m,n))$ that they conjecture is
isomorphic to the equivariant quantum $K$-theory ring
$\QK_T(\Gr(m,n))$ of a Grassmannian of type A
\cite{gorbounov.korff:quantum}.  In our last section we use
Molev-Sagan equations \cite{molev.sagan:littlewood-richardson,
  knutson.tao:puzzles} to prove, for any flag variety $X$, that the
structure of the ring $\QK_T(X)$, including its Schubert structure
constants and the underlying Gromov-Witten invariants of $X$, is
uniquely determined by products involving divisors.  Since Gorbounov
and Korff prove that that their algebra $qh^*(\Gr(m,n))$ satisfies the
same Chevalley formula as the one proved for equivariant quantum
$K$-theory in this paper, we obtain a proof of their conjecture.  In
particular, a presentation and a Giambelli formula from
\cite{gorbounov.korff:quantum} for the algebra $qh^*(\Gr(m,n))$ are
also valid in equivariant quantum $K$-theory.

Our paper is organized as follows.  In \Section{qk} we recall
the definition of quantum $K$-theory for arbitrary flag varieties.
\Section{cominuscule} covers quantum $K$-theory of cominuscule
varieties and proves the quantum part of our Chevalley formula.  Two
equivalent versions of this formula are given in \Theorem{qktchev} and
\Corollary{chev_const}, which are followed by examples.
\Section{chevproof} contains our geometric proof of the Chevalley
formula for the equivariant $K$-theory of cominuscule varieties.
Finally, \Section{ms} proves that the equivariant quantum $K$-theory
of any flag variety is uniquely determined by products with divisor
classes, see \Proposition{unique_sol_qkt}.

This project was started while the authors visited the University of
Copenhagen during the Summer of 2014.  We thank the Mathematics
Department in Copenhagen for their hospitality and for providing a
friendly and stimulating environment.  We also thank Ionu{\c t}
Ciocan-Fontanine for bringing to our attention
\cite[Lemma~4.1.3]{ciocan-fontanine.kim.ea:abeliannonabelian} and an
anonymous referee for valuable suggestions.

\section{Quantum $K$-theory of flag varieties}
\label{sec:qk}

\subsection{Equivariant $K$-theory}

All algebraic varieties and schemes in this paper will be defined over
the complex numbers $\C$.  Given a variety $X$ with an action of an
algebraic torus $T = (\C^*)^n$, we let $K_T(X)$ denote the
Grothendieck group of $T$-equivariant coherent $\cO_X$-modules.  This
group is a module over the Grothendieck ring $K^T(X)$ of
$T$-equivariant vector bundles over $X$.  If $X$ is non-singular, then
the implicit map $K^T(X) \to K_T(X)$ that sends a vector bundle to its
sheaf of sections is an isomorphism.  Our main references for
equivariant $K$-theory are
\cite[Ch.~5]{chriss.ginzburg:representation} and \cite[\S
15.1]{fulton:intersection}.  A shorter summary of the main properties
can also be found in \cite[\S 3]{buch.mihalcea:quantum}.  Pullback
along the structure morphism $X \to \{\pt\}$ gives $K^T(X)$ a
structure of algebra over the ring $\Gamma = K^T(\pt)$.  This ring
$\Gamma$ is simply the ring of virtual representations of $T$.  Given
any character $\al : T \to \C^*$ we let $\C_\al$ denote the
one-dimensional representation of $T$ defined by $t.z = \al(t)z$ for
$t \in T$ and $z \in \C_\al$.  The ring $\Gamma$ has a $\Z$-basis
consisting of the classes $[\C_\al]$ of these representations.  When
$X$ is projective we let $\euler{X} : K_T(X) \to \Gamma$ denote the
pushforward along the structure morphism of $X$.  This is a
homomorphism of $\Gamma$-modules by the projection formula.

\subsection{Schubert varieties}

Let $X = G/P$ be a flag variety defined by a complex semisimple linear
algebraic group $G$ and a parabolic subgroup $P$.  Fix a maximal torus
$T$ and a Borel subgroup $B$ such that
$T \subset B \subset P \subset G$, and let $B^- \subset G$ be the
opposite Borel subgroup defined by $B \cap B^- = T$.  Let
$W = N_G(T)/T$ be the Weyl group of $G$, $W_P = N_P(T)/T$ the Weyl
group of $P$, and let $W^P \subset W$ be the set of minimal
representatives of the cosets in $W/W_P$.  Each element $w \in W$
defines a $B$-stable Schubert variety $X_w = \ov{Bw.P}$ and an
(opposite) $B^-$-stable Schubert variety $X^w = \ov{B^-w.P}$ in $X$.
When $w \in W^P$ is a minimal representative we have
$\dim(X_w) = \codim(X^w,X) = \ell(w)$, where $\ell(w)$ denotes the
length of $w$.  A product $w = u_1 u_2 \cdot \ldots \cdot u_m$ of
elements in $W$ will be called {\em reduced\/} if
$\ell(w) = \sum \ell(u_i)$.  This implies that any consecutive
subproduct $u_i u_{i+1} \cdot \ldots \cdot u_j$ is also reduced.  The
dual element of $u \in W^P$ is $u^\vee = w_0 u w_P$, where $w_0$ is
the longest element in $W$ and $w_P$ is the longest element in $W_P$.
Notice that $1^\vee$ is the longest element in $W^P$.  Let $\Phi$ be
the root system of $(G,T)$, interpreted as a set of characters of $T$.
Let $\Phi^+$ be the set of positive roots determined by $B$, let
$\Delta \subset \Phi^+$ be the simple roots, and let
$\Delta_P \subset \Delta$ be the set of simple roots $\be$ for which
the associated reflection $s_\be$ belongs to $W_P$.  We need the
following topological fact.

\begin{lemma}\label{lemma:tstable}
  Let $Z \subset X$ be a $T$-stable closed subvariety and let
  $x_0 \in X^T$ be a $T$-fixed point such that
  $Z \cap B.x_0 \neq \emptyset$.  Then $x_0 \in Z$.
\end{lemma}
\begin{proof}
  Choose any group homomorphism $\phi : \C^* \to T$ such that, for
  each $\al \in \Phi^+$ the composition $\al\, \phi : \C^* \to \C^*$
  is given by $\al(\phi(s)) = s^m$ for some $m > 0$.  Let
  $U_\al \subset G$ be the $T$-stable subgroup with
  $\Lie(U_\al) = \Lie(G)_\al$.  It follows from
  \cite[Thm.~26.3]{humphreys:linear} that
  $\lim_{s \to 0} \phi(s) b \phi(s)^{-1} = 1$ for any element
  $b \in U_\al$, $\al \in \Phi^+$.  Since $B$ is generated by $T$ and
  the subgroups $U_\al$ for $\al \in \Phi^+$, we deduce that
  $\lim_{s \to 0} \phi(s) b \phi(s)^{-1} \in T$ for any $b \in B$.
  Now choose $b \in B$ such that $b.x_0 \in Z$.  We then obtain
  $x_0 = \left(\lim_{s \to 0} \phi(s) b \phi(s)^{-1} \right).x_0 =
  \lim_{s \to 0} \phi(s) b.x_0 \in Z$, as required.
\end{proof}

\subsection{Schubert structure constants}

The equivariant $K$-theory ring $K^T(X)$ of the flag variety $X$ has a
basis over $\Gamma$ consisting of the (opposite) {\em Schubert
  classes\/} $\cO^w = [\cO_{X^w}]$ for $w \in W^P$.  The classes
$\cO_w = [\cO_{X_w}]$ form an alternative basis.  The {\em Schubert
  structure constants\/} of $K^T(X)$ are the classes
$N^{w,0}_{u,v} \in \Gamma$ defined for $u,v,w \in W^P$ by the identity
\[
\cO^u \cdot \cO^v = \sum_w N^{w,0}_{u,v}\, \cO^w \,.
\]
Let $\cO_w^\vee \in K^T(X)$ denote the basis element dual to $\cO^w$,
defined by $\euler{X}(\cO^u \cdot \cO^\vee_w) = \delta_{u,w}$ for
$u,w \in W^P$.  We then have
$N^{w,0}_{u,v} = \euler{X}(\cO^u \cdot \cO^v \cdot \cO^\vee_w)$.  The
{\em boundary\/} of the Schubert variety $X_w$ is the closed
subvariety defined by $\partial X_w = X_w \ssm Bw.P$.  Brion has
proved the identity
\begin{equation}\label{eqn:brion_dual}
  \cO_w^\vee = [I_{\partial X_w}] \,,
\end{equation}
where $I_{\partial X_w} \subset \cO_{X_w}$ denotes the ideal sheaf of
this boundary \cite{brion:positivity}.

Calculations in the ring $K^T(X)$ may be carried out by utilizing that
restriction to the set of $T$-fixed points
$X^T = \{ w.P \mid w \in W^P \}$ provides an injective ring
homomorphism $K^T(X) \to K^T(X^T) = \prod_{w \in W^P} \Gamma$
\cite{kostant.kumar:t-equivariant} (see also
\cite{goresky.kottwitz.ea:equivariant}).  The images of Schubert
classes under this map are given by the restriction formulas in
\cite{andersen.jantzen.ea:representations, billey:kostant,
  graham:equivariant, willems:k-theorie} (see also
\cite{knutson:schubert}).  The ring structure of $K^T(X)$ is also
determined by the Chevalley formula of Lenart and Postnikov
\cite{lenart.postnikov:affine}, which provides an explicit expression
for the product of any Schubert class with a divisor in $K^T(X)$.  A
version of this formula for cominuscule varieties will be discussed in
\Section{chevalley}.  Other useful formulas for structure constants in
various special cases can be found in e.g.\
\cite{buch.samuel:k-theory, knutson:puzzles, buch:mutations,
  pechenik.yong:equivariant} and the references therein.

\subsection{Quantum $K$-theory}

A homology class $d = \sum d_\be [X_{s_\be}] \in H_2(X;\Z)$, with the
sum over $\be \in \Delta \ssm \Delta_P$, is called an {\em effective
  degree\/} if $d_\be \geq 0$ for each $\be$.  For $d,e \in H_2(X;\Z)$
we write $e \leq d$ if and only if $d-e$ is effective.  For any
effective degree $d$ we let $\Mb_{0,n}(X,d)$ denote the Kontsevich
moduli space of $n$-pointed stable maps to $X$ of genus zero and
degree $d$ \cite{fulton.pandharipande:notes}.  This space is equipped
with evaluation maps $\ev_i : \Mb_{0,n}(X,d) \to X$ for
$1 \leq i \leq n$.  Given classes
$\sigma_1, \sigma_2, \dots, \sigma_n \in K^T(X)$, define a
corresponding (equivariant $K$-theoretic) Gromov-Witten invariant of
$X$ by
\[
I_d(\sigma_1,\sigma_2,\dots,\sigma_n) \ = \
\euler{\Mb_{0,n}(X,d)}\left( \ev_1^*(\sigma_1) \cdot \ev_2^*(\sigma_2)
  \cdot \ldots \cdot \ev_n^*(\sigma_n)\right) \ \in \Gamma \,.
\]
This invariant is $\Gamma$-linear in each argument $\sigma_i$.

The (small) $T$-equivariant quantum $K$-theory ring of $X$ is an
algebra $\QK_T(X)$ over the ring of formal power series
$\Gq = \Gamma\llbracket q_\be : \be \in \Delta \ssm
\Delta_P\rrbracket$,
which has one variable $q_\be$ for each simple root
$\be \in \Delta \ssm \Delta_P$.  As a module over $\Gq$ we have
$\QK_T(X) = K^T(X) \otimes_\Gamma \Gq$.  In other words, $\QK_T(X)$ is
a free module over $\Gq$ with a basis consisting of the Schubert
classes $\cO^w$ for $w \in W^P$.  For $d = \sum d_\be [X_{s_\be}]$ we
write $q^d = \prod_\be q_\be^{d_\be}$.  Multiplication in $\QK_T(X)$
is defined by
\[
\cO^u \star \cO^v = \sum_{w,d \geq 0} N^{w,d}_{u,v}\, q^d\, \cO^w
\]
where the sum is over all effective degrees $d$ and $w \in W^P$.  The
structure constants $N^{w,d}_{u,v}$ are defined recursively by
\begin{equation}\label{eqn:Nuvwd}
  N^{w,d}_{u,v} \ = \ I_d(\cO^u, \cO^v, \cO_w^\vee) \ - \
  \sum_{\kappa,\,0<e\leq d} N^{\kappa,d-e}_{u,v}\, I_e(\cO^\kappa,\cO_w^\vee)
\end{equation}
where this sum is over all $\kappa \in W^P$ and degrees $e$ for which
$0 < e \leq d$.  A theorem of Givental states that this defines an
associative product \cite{givental:wdvv} (see also
\cite{witten:verlinde, ruan.tian:mathematical,
  kontsevich.manin:gromov-witten, lee:quantum*1}).

The ring $\QK_T(X)$ is a formal deformation of the equivariant
$K$-theory ring $K^T(X)$.  It is best understood when $X$ is a
cominuscule variety.  This case will be discussed in
\Section{cominuscule}.  When the Picard group of $X$ has rank one, it
has been proved in \cite{buch.chaput.ea:finiteness,
  buch.chaput.ea:rational} that the product $\cO^u \star \cO^v$
contains only finitely many non-zero terms.  It is an open question if
this holds for general flag varieties.  Combinatorial models for the
equivariant quantum $K$-theory of complete flag manifolds $G/B$ have
been defined by Lenart and Maeno \cite{lenart.maeno:quantum} and by
Lenart and Postnikov \cite{lenart.postnikov:affine}.  These models are
correct for $\bP^1$ \cite{buch.mihalcea:quantum} and consistent with
the computation of the quantum $K$-theory of $\SL(3)/B$
\cite{iritani.milanov.ea:reconstruction}, but for all other spaces it
is an open problem to prove that the models agree with the geometric
definition of $\QK_T(X)$.  In addition, the lack of functoriality of
quantum $K$-theory makes it unclear how these models relate to
cominuscule varieties.  Another model for the equivariant quantum
$K$-theory of Grassmannians of type A has been conjectured by
Gorbounov and Korff \cite{gorbounov.korff:quantum}.  This model is a
consequence of the results proved in this paper (see \Remark{gk}).
Finally it is expected that the structure constants of $\QK_T(X)$
satisfy the following positivity property \cite{griffeth.ram:affine,
  lenart.maeno:quantum, lenart.postnikov:affine,
  buch.mihalcea:quantum}.

\begin{conj}
  Let $u,v,w \in W^P$, let $d \in H_2(X;\Z)$ be an effective degree,
  and set $s = \ell(u)+\ell(v)+\ell(w)+\int_d c_1(T_X)$.  Then we have
  \[
  (-1)^s\, N^{w,d}_{u,v} \,\in\, \N\big[\, [\C_{-\be}]-1 : \be \in
  \Delta \,\big] \,.
  \]
\end{conj}

In other words, the structure constant $N^{w,d}_{u,v}$ can be written,
up to a sign, as a polynomial with non-negative integer coefficients
in the classes $[\C_{-\be}]-1$ for $\be \in \Delta$.  This has been
proved for the structure constants $N^{w,0}_{u,v}$ of $K^T(X)$ by
Anderson, Griffeth, and Miller \cite{anderson.griffeth.ea:positivity}.
Earlier it was proved by Brion \cite{brion:positivity} that the
structure constants of the ordinary $K$-theory ring $K(X)$ have signs
that alternate with codimension.

\subsection{Curve neighborhoods}\label{sec:nbhd}

Given a closed subvariety $\Omega \subset X$ and an effective degree
$d \in H_2(X;\Z)$, we let $\Gamma_d(\Omega)$ denote the closure of the
union of all rational curves of degree $d$ in $X$ that pass through
$\Omega$.  Equivalently we have
$\Gamma_d(\Omega) = \ev_2(\ev_1^{-1}(\Omega))$, where
$\ev_1, \ev_2 : \Mb_{0,2}(X,d) \to X$ are the evaluation maps.  It was
proved in \cite{buch.chaput.ea:finiteness} that $\Gamma_d(\Omega)$ is
irreducible whenever $\Omega$ is irreducible.  In particular, if
$\Omega$ is a $B$-stable Schubert variety in $X$ then so is
$\Gamma_d(\Omega)$.  Given $w \in W^P$ we may therefore define related
Weyl group elements $w(d)$ and $w(-d)$ in $W^P$ by the identities
$\Gamma_d(X_w) = X_{w(d)}$ and $\Gamma_d(X^w) = X^{w(-d)}$.  We have
$w(-d) \leq w \leq w(d)$ in the Bruhat order of $W^P$, and since $X^w$
is a translate of $X_{w^\vee}$, we obtain $w(-d)^\vee = w^\vee(d)$.  A
combinatorial description of $w(d)$ and $w(-d)$ can be found in
\cite{buch.chaput.ea:finiteness} or \Section{comin_nbhd} when $X$ is
cominuscule and in \cite{buch.mihalcea:curve} for general flag
varieties.

It was proved in \cite{buch.chaput.ea:finiteness} that
$I_d(\cO^w) = 1$, i.e.\ any single-pointed Gromov-Witten invariant of
a Schubert class is equal to one.  More generally, it follows from
\cite[Prop.~3.2]{buch.chaput.ea:finiteness} that two-pointed
Gromov-Witten invariants are given by the identity
\[
I_d(\cO^u, \cO_w^\vee) \ = \ \euler{X}(\cO^{u(-d)} \cdot \cO_w^\vee) \
= \ \delta_{u(-d),w} \,.
\]
As a consequence, the definition \eqn{Nuvwd} of the structure
constants of $\QK_T(X)$ is equivalent to the identity
\begin{equation}\label{eqn:N2gw}
  I_d(\cO^u, \cO^v, \cO_w^\vee) \ = \
  \sum_{0 \leq e \leq d} \sum_{\ \kappa :\, \kappa(-e)=w} N^{\kappa,d-e}_{u,v} \,.
\end{equation}

For any classes $\sigma_1, \sigma_2 \in K^T(X)$ we define the power
series
\[
\sigma_1 \odot \sigma_2 \ = \ \sum_{w, d \geq 0} I_d(\sigma_1,
\sigma_2, \cO_w^\vee)\, q^d\, \cO^w
\]
in $\QK_T(X)$.  We also define an automorphism
$\Psi : \QK_T(X) \to \QK_T(X)$ of $\Gq$-modules by
\[
\Psi(\cO^w) \ = \ \sum_{e \geq 0}\, q^e\, \cO^{w(-e)} \,.
\]
Equation \eqn{N2gw} is equivalent to the following statement.

\begin{prop}\label{prop:Psi}
  For $\sigma_1, \sigma_2 \in K^T(X)$ we have
  $\sigma_1 \odot \sigma_2 = \Psi(\sigma_1 \star \sigma_2)$ in
  $\QK_T(X)$.
\end{prop}

\section{Cominuscule varieties}\label{sec:cominuscule}

\subsection{Bruhat order}

A simple root $\ga \in \Delta$ is called {\em cominuscule\/} if, when
the highest root is written as a linear combination of simple roots,
the coefficient of $\ga$ is one.  The flag variety $X = G/P$ is called
cominuscule if $\Delta \ssm \Delta_P$ consists of a single cominuscule
root $\ga$.  If in addition the root system $\Phi$ is simply laced,
then $X$ is also called {\em minuscule}.  This holds for all the
cominuscule varieties listed in the introduction, except for
Lagrangian Grassmannians and quadrics of odd dimension.  We will
assume that $X$ is cominuscule in this section.  In this case it was
proved by Proctor that the Bruhat order on $W^P$ is a distributive
lattice that agrees with the left weak Bruhat order
\cite{proctor:bruhat}.  Stembridge has proved that all elements of
$W^P$ are {\em fully commutative}, which means that any reduced
expression for an element of $W^P$ can be obtained from any other by
interchanging commuting simple reflections.  We proceed to summarize
the facts we need in more detail.  Proofs of our claims can be found
in \cite{proctor:bruhat, stembridge:fully, perrin:small*1,
  buch.samuel:k-theory}.

The root lattice $\Span_\Z(\Delta)$ has a partial order defined by
$\al' \leq \al$ if and only if $\al-\al'$ can be written as a sum of
positive roots.  Let $\cP_X = \{ \al \in \Phi \mid \al \geq \ga \}$ be
the set of positive roots $\al$ for which the coefficient of the
cominuscule root $\ga$ is one, with the induced partial order.  For
any element $u \in W$ we let
$I(u) = \{ \al \in \Phi^+ \mid u.\al < 0 \}$ denote the inversion set
of $u$.  We then have $\ell(u) = |I(u)|$.  The assignment
$u \mapsto I(u)$ restricts to a bijection between the elements of
$W^P$ and the lower order ideals of $\cP_X$.  This assignment is order
preserving in the sense that $u \leq v$ if and only if
$I(u) \subset I(v)$.  Given a lower order ideal $\la \subset \cP_X$,
let $\la = \{ \al_1, \al_2, \dots, \al_{|\la|} \}$ be any ordering of
its elements that is compatible with the partial order $\leq$, i.e.\
$\al_i < \al_j$ implies $i<j$.  Then the element of $W^P$
corresponding to $\la$ is the product of reflections
$w_\la = s_{\al_1} s_{\al_2} \dots s_{\al_{|\la|}}$.  The order ideal
$\la$ is a generalization of the Young diagrams and shifted Young
diagrams known from the Schubert calculus of the classical
Grassmannians.  For this reason the roots in $\cP_X$ will sometimes be
called {\em boxes}, and the order ideal $I(u)$ will be called the {\em
  shape\/} of $u$.

Given two elements $u, w \in W^P$ with $u \leq w$, we will use the
notation $w/u = w u^{-1} \in W$.  Since the Bruhat order on $W^P$
agrees with the left weak Bruhat order, we have
$\ell(w/u) = \ell(w)-\ell(u)$.  For any root $\al \in \cP_X$, consider
the order ideal $\la = \{ \al' \in \cP_X \mid \al' < \al \}$ of roots
that are smaller than $\al$, and set $\delta(\al) = w_\la.\al$.  Then
$s_{\delta(\al)} = w_\la s_\al w_\la^{-1} = w_{\la \cup \al}/w_\la$
has length one.  It follows that $\delta : \cP_X \to \Delta$ is a
labeling of the boxes in $\cP_X$ by simple roots.  In addition, the
element $w/u$ depends only on the {\em skew shape\/} $I(w) \ssm I(u)$.
More precisely, if
$I(w) \ssm I(u) = \{ \al_1, \al_2, \dots, \al_\ell \}$ is any ordering
compatible with $\leq$, then
$w/u = s_{\delta(\al_\ell)} \cdots s_{\delta(\al_2)}
s_{\delta(\al_1)}$
is a reduced expression for $w/u$.  In the special case $u=1$, every
reduced expression for $w$ can be obtained in this way.

We will say that $w/u$ is a {\em rook strip\/} if this element of $W$
is a product of commuting simple reflections.  Equivalently, no pair
of roots in $I(w)\ssm I(u)$ are comparable by the order $\leq$.  We
call $w/u$ a {\em short rook strip\/} if it is a product of commuting
reflections defined by short simple roots, i.e.\ $I(w) \ssm I(u)$
consists of incomparable short roots.  Notice that if the root system
$\Phi$ is simply laced, then all roots are long by convention, so
$w/u$ is a short rook strip if and only if $w=u$.

The Bruhat order on $W^P$ is a distributive lattice with operations
$u \cap v$ and $u \cup v$ for $u, v \in W^P$.  These operations are
defined combinatorially by $I(u \cap v) = I(u) \cap I(v)$ and
$I(u \cup v) = I(u) \cup I(v)$, and geometrically by
$X_{u \cap v} = X_u \cap X_v$ and $X^{u \cup v} = X^u \cap X^v$.
Notice that for any $u, z \in W^P$ we have
$(u \cup z)/z = u/(u \cap z)$.  Since dualization of Weyl group
elements is an order-reversing involution of $W^P$, we also have
$(u \cap v)^\vee = u^\vee \cup v^\vee$ and
$(u \cup v)^\vee = u^\vee \cap v^\vee$.

\subsection{Curve neighborhoods}\label{sec:comin_nbhd}

Since $P$ is a maximal parabolic subgroup of $G$ it follows that
$H_2(X;\Z) = \Z$.  For $u \in W^P$ and $d \in \N$ we let
$u(d), u(-d) \in W^P$ be the unique Weyl group elements for which
$\Gamma_d(X_u) = X_{u(d)}$ and $\Gamma_d(X^u) = X^{u(-d)}$.  Since $X$
is cominuscule, we have $u(d)(d') = u(d+d')$ and
$u(-d)(-d') = u(-d-d')$ for any effective degrees $d, d' \in \N$, see
\cite{chaput.manivel.ea:quantum*1} or
\cite[Lemma~4.2]{buch.chaput.ea:finiteness}.  However, it is easy to
find examples where both $u(d')(-d)$ and $u(-d)(d')$ are different
from $u(d'-d)$.  It was proved in
\cite[Lemma~4.4]{buch.chaput.ea:finiteness} that $u(1)$ is the minimal
representative of the coset $u w_P s_\ga W_P$ whenever $X_u \neq X$.
Equivalently, $u(-1)$ is the minimal representative of $u s_\ga W_P$
whenever $u \neq 1$.

For $d \geq 0$ we let $z_d \in W^P$ be the unique element such that
$X_{z_d} = \Gamma_d(1.P)$.  This element satisfies that $z_d w_P$ is
inverse to itself, where $w_P$ denotes the longest element in $W_P$.
In fact, if we let $C \subset X$ be any stable curve of degree $d$
from $1.P$ to $z_d.P$, then $(z_d w_P)^{-1}.C$ is a curve of the same
degree from $1.P$ to $(z_d w_P)^{-1}.P$, so we obtain
$(z_d w_P)^{-1}.P \in \Gamma_d(1.P)$.  This implies that
$(z_d w_P)^{-1} \leq z_d w_P$ and therefore
$(z_d w_P)^{-1} = z_d w_P$.  It follows that $z_d w_P$ is the maximal
representative of the {\em left\/} coset $W_P z_d w_P$.  If
$X_{z_d} \neq X$, we deduce that $\cP_X \ssm I(z_d)$ contains a unique
minimal box $\al_{d+1}$, and this box satisfies
$\delta(\al_{d+1}) = \ga$.  In other words we have
$I(z_d) = \{ \al \in \cP_X \mid \al \not\geq \al_{d+1} \}$.  Notice
also that since $z_{d+1} = z_d(1)$ is the minimal representative of
the coset $z_d w_P s_\ga W_P$, we may choose $x \in W_P$ such that
$z_{d+1} w_P = z_d w_P s_\ga x^{-1}$ holds as an identity of reduced
products.  This implies that
$z_{d+1} w_P = (z_{d+1} w_P)^{-1} = x s_\ga (z_d w_P)^{-1} = x s_\ga
z_d w_P$,
hence $z_{d+1} = x s_\ga z_d$ expresses $z_{d+1}$ as a reduced
product.  Since every reduced expression of $z_{d+1}$ corresponds to
an ordering of the boxes of $I(z_{d+1})$, this implies that
$\al_{d+1}$ is the only box of $I(z_{d+1}) \ssm I(z_d)$ that is sent
to $\ga$ by the map $\delta$.  In fact, we have
$\delta^{-1}(\ga) = \{ \al_1 < \al_2 < \dots < \al_{d_X(2)} \}$ where
$d_X(2)$ is the smallest degree of a rational curve joining two
general points in $X$.

\begin{lemma}\label{lemma:comin_nbhd}
  For $u \in W^P$ and $d \in \N$ we have
  $u(-d) = u/(u \cap z_d) = (u \cup z_d)/z_d$ and
  $u(d) = (u \cap z_d^\vee)z_d$.
\end{lemma}
\begin{proof}
  For $d=0$ the lemma follows because $u(0) = u$ and $z_0 = 1$.  Let
  $d>0$ and assume by induction that $u(1-d) = u/(u \cap z_{d-1})$.
  Notice that, if $u(1-d)=1$, then the lemma holds for $d$ because
  $u \leq z_{d-1} \leq z_d$ and $u(-d)=1$.  We may therefore assume
  that $u(1-d) \neq 1$, or equivalently $\al_d \in I(u)$.

  Set $v = u/(u \cap z_d) = (u \cup z_d)/z_d$.  Then we have an
  identity of reduced products $v z_d w_P = (u \cup z_d) w_P$.  Since
  $z_d w_P$ is maximal in its left coset $W_P z_d w_P$, it follows
  that $v$ is minimal in its right coset $v W_P$, in other words
  $v \in W^P$.  Since
  $I(u \cap z_d) \ssm I(u \cap z_{d-1}) \subset I(z_d) \ssm
  I(z_{d-1})$
  and $\al_d \in I(u)$, it follows that
  $(u \cap z_d)/(u \cap z_{d-1}) = y s_\ga$ for some $y \in W_P$.
  Since
  $u(1-d) = u/(u \cap z_{d-1}) = v \, (u \cap z_d)/(u\cap z_{d-1}) = v
  y s_\ga$,
  we obtain $v = u(1-d) s_\ga y^{-1}$.  It follows that $v \in W^P$ is
  the minimal representative of the coset $u(1-d) s_\ga W_P$, as
  required.

  The last identity now follows because
  $u^\vee(d) = u(-d)^\vee = w_0 (u \cup z_d) z_d^{-1} w_P = (u \cup
  z_d)^\vee (z_d w_P)^{-1} w_P = (u^\vee \cap z_d^\vee) z_d$.
\end{proof}

\def\vmm#1{\vspace{#1mm}}
\begin{table}
\caption{Partially ordered sets of cominuscule varieties
  with $I(z_1)$ highlighted.}
\label{tab:tablez1}
\begin{tabular}{|c|c|}
\hline
&\vmm{-2}\\
Grassmannian $\Gr(3,7)$ of type A & Max.\ orthog.\ Grassmannian $\OG(6,12)$
\\
&\vmm{-3}\\
\pic{1}{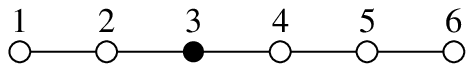} &\\
&\\
&\vmm{-3}\\
$\tableau{12}{
[aLlTt]3 & [aTtBb]4 & 5 & [aTtBbRr]6 \\
[aLlRr]2 & [a]3 & 4 & 5 \\
[aLlRrBb]1 & [a]2 & 3 & 4
}$
&\vmm{-27}\\
&\pic{1}{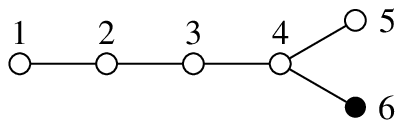}
\\ &
$\tableau{12}{
[aLlTtBb]6 & [aTt]4 & 3 & 2 & [aTtRr]1 \\
  & [aLlBb]5 & [aBb]4 & 3 & [aBbRr]2 \\
  &   & [a]6 & 4 & 3 \\
  &   &   & 5 & 4 \\
  &   &   &   & 6
}$
\\ &
\vmm{-2}\\
\hline
&\vmm{-2}\\
Lagrangian Grassmannian $\LG(6,12)$ & Cayley Plane $E_6/P_6$
\\
&\vmm{-3}\\
\pic{1}{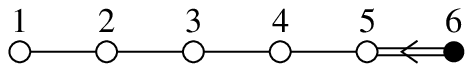} & \\
&\vmm{-2}\\
$\tableau{12}{
[aLlTtBb]6 & [aTtBb]5 & 4 & 3 & 2 & [aTtBbRr]1 \\
  & [a]6 & 5 & 4 & 3 & 2 \\
  &   & 6 & 5 & 4 & 3 \\
  &   &   & 6 & 5 & 4 \\
  &   &   &   & 6 & 5 \\
  &   &   &   &   & 6
}$
&\vmm{-36}\\
& \pic{1}{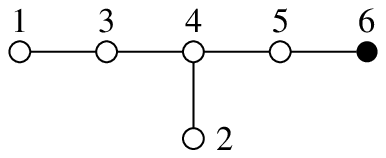} \\
&\vmm{-2}\\
&
$\tableau{12}{
[aLlTtBb]6 & [aTtBb]5 & [aTt]4 & 3 & [aTtRr]1 \\
  &   & [aLlBb]2 & [a]4 & [aRr]3 \\
  &   &   & [aLlBb]5 & [aBb]4 & [aTtBbRr]2 \\
  &   &   & [a]6 & 5 & 4 & 3 & 1
}$
\\
& \vmm{-2}\\
\hline
& \vmm{-2}\\
Even quadric $Q^{10} \subset \bP^{11}$ & Freudenthal variety $E_7/P_7$
\\
&\vmm{-3}\\
\pic{1}{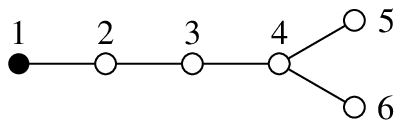} &
\\
&\vmm{-3}\\
$\tableau{12}{
[aLlTtBb]1 & [aTtBb]2 & 3 & [aTt]4 & [aTtRr]5 \\
  &   &   & [aLlBb]6 & [aBb]4 & [aTtBb]3 & [aTtBbRr]2 & [a]1
}$
& \\
&\\
\hhline{-~}
&\vmm{-2}\\
Odd quadric $Q^{11} \subset \bP^{12}$ &
\\
&\vmm{-3}\\
\pic{1}{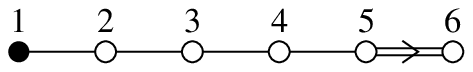} & \\
&\vmm{-1}\\
$\tableau{12}{
[aLlTtBb]1 & [aTtBb]2 & 3 & 4 & 5 & 6 & 5 & 4 & 3 & [aTtBbRr]2 & [a]1
}$
& \vmm{-52}\\
& \pic{1}{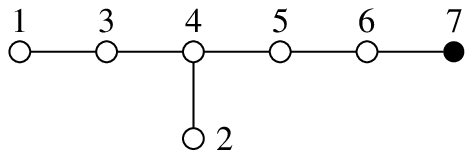} \\
&\vmm{-3}\\
&
$\tableau{12}{
[aLlTtBb]7 & [aTtBb]6 & 5 & [aTt]4 & 3 & [aTtRr]1 \\
  &   &   & [aLlBb]2 & [a]4 & [aRr]3 \\
  &   &   &   & [aLl]5 & [a]4 & [aTtRr]2 \\
  &   &   &   & [aLlBb]6 & [aBb]5 & 4 & [aTtBb]3 & [aTtBbRr]1 \\
  &   &   &   & [a]7 & 6 & 5 & 4 & 3 \\
  &   &   &   &   &   &   & 2 & 4 \\
  &   &   &   &   &   &   &   & 5 \\
  &   &   &   &   &   &   &   & 6 \\
  &   &   &   &   &   &   &   & 7
}$
\vmm{-2}\\
& \\
\hline
\end{tabular}
\end{table}

\Table{tablez1} displays one cominuscule variety $X$ from each family
together with the associated Dynkin diagram and the partially ordered
set $\cP_X$.  The marked node in the Dynkin diagram indicates the
cominuscule simple root $\ga$.  The roots of $\cP_X$ are represented
as boxes, and the partial order is given by $\al' \leq \al$ if and
only if $\al'$ is located north-west of $\al$.  Each box
$\al \in \cP_X$ is labeled by the number of the simple root
$\delta(\al)$.  In addition the shape $I(z_1)$ is marked.
\Lemma{comin_nbhd} implies that, given any element $u \in W^P$, one
may obtain the shape of $u(-1)$ from the shape of $u$ by first
removing any boxes contained in $I(z_1)$, and then moving the
remaining boxes to the upper-left corner of $\cP_X$.  More details
about some of the cases can be found in \Example{gr37}, \Example{LG},
and \Section{oddquadric}.

\subsection{Quantum $K$-theory}

Since $X$ is cominuscule, the equivariant quantum $K$-theory ring
$\QK_T(X)$ is an algebra over the power series ring $\Gq$ in a single
variable $q_\ga$; we drop the subscript and denote this variable by
$q$.  We proceed to give a simplified construction of the product in
$\QK_T(X)$ in the cominuscule case.  Let $\psi : K^T(X) \to K^T(X)$ be
the homomorphism of $\Gamma$-modules defined by
$\psi(\cO^u) = \cO^{u(-1)}$.  This map can also be defined
geometrically by $\psi = (\ev_2)_*(\ev_1)^*$, where $\ev_1$ and
$\ev_2$ are the evaluation maps from the moduli space
$\Mb_{0,2}(X,1)$.  We extend $\psi$ by linearity to a homomorphism
$\psi : \QK_T(X) \to \QK_T(X)$ of $\Gq$-modules.

\begin{prop}\label{prop:comin_prod}
  For $\sigma_1, \sigma_2 \in K^T(X)$ we have
  $\sigma_1 \star \sigma_2 = (1 - q \psi)(\sigma_1 \odot \sigma_2)$ in
  $\QK_T(X)$.
\end{prop}
\begin{proof}
  The identity $w(-e-1) = w(-e)(-1)$ among Weyl group elements implies
  that the operator $\Psi$ from \Section{nbhd} is given by
  $\Psi = 1 + q\psi + q^2\psi^2 + \cdots$.  It therefore follows from
  \Proposition{Psi} that
  $\sigma_1 \star \sigma_2 = \Psi^{-1}(\sigma_1 \odot \sigma_2) = (1 -
  q\psi)(\sigma_1 \odot \sigma_2)$.
\end{proof}

Let $\Gamma_d(X_u, X^v) \subset X$ denote the union of all stable
curves of degree $d$ in $X$ that pass through $X_u$ and $X^v$.
Equivalently we have
$\Gamma_d(X_u,X^v) = \ev_3(\ev_1^{-1}(X_u) \cap \ev_2^{-1}(X^v))$,
where $\ev_1, \ev_2, \ev_3 : \Mb_{0,3}(X,d) \to X$ are the evaluation
maps.  The following version of the {\em quantum equals classical\/}
theorem was proved in \cite[Cor.~4.2]{buch.chaput.ea:projected}.

\begin{thm}\label{thm:qclas}
  For $u, v \in W^P$ we have in $K^T(X)$ that
  \[
  [\cO_{\Gamma_d(X_u, X^v)}] \ = \sum_{w \in W^P} I_d(\cO_u, \cO^v,
  \cO_w^\vee)\, \cO^w \,.
  \]
\end{thm}

As a consequence we have
$\cO_u \odot \cO^v = \sum_{d \geq 0} [\cO_{\Gamma_d(X_u,X^v)}]\, q^d$
for all $u,v \in W^P$.

\subsection{Chevalley formula for $K^T(X)$}\label{sec:chevalley}

Let $J = 1 - \cO^{s_\ga} \in K^T(X)$ denote the class of the ideal
sheaf of the (opposite) Schubert divisor $X^{s_\ga}$.  Our main result
is an explicit combinatorial formula for any product of the form
$J \star \cO^u$ in the ring $\QK_T(X)$.  For $u \in W^P$ we let
$J_u = J|_{u.P} \in \Gamma$ denote the restriction of the class $J$ to
the $T$-fixed point $u.P \in X$.  For $\al \in \Phi$ we let
$\al^\vee = \frac{2\al}{(\al,\al)}$ denote the coroot of $\al$, and we
let $\omega_\ga$ be the fundamental weight corresponding to $\gamma$.
The following lemma is proved in \Section{minuscule}.

\begin{lemma}\label{lemma:Ju}
  We have $J_u = [\C_{u.\omega_\ga - \omega_\ga}] \in \Gamma$, and the
  weight satisfies the identity
  \[
  \omega_\ga - u.\omega_\ga = \sum_{\al \in I(u)} (\omega_\ga,
  \al^\vee)\, \delta(\al) \,.
  \]
\end{lemma}

Define a homomorphism $\theta_0 : K^T(X) \to K^T(X)$ of
$\Gamma$-modules by
\[
\theta_0(\cO^u) \ = \
\sum_{w/u \text{ rook strip}} (-1)^{\ell(w/u)}\, \cO^w
\]
where the sum is over all $w \in W^P$ for which $u \leq w$ and $w/u$
is a rook strip.  Equivalently, $\theta_0(\cO^u)$ is the class of the
ideal sheaf of the boundary $\partial X^u \subset X^u$ by Brion's
identity \eqn{brion_dual} together with the following lemma.

\begin{lemma}\label{lemma:th0dual}
  The class $\theta_0(\cO^u)$ is dual to $\cO_u$ in $K^T(X)$.  More
  precisely we have
  $\euler{X}(\theta_0(\cO^u) \cdot \cO_v) = \delta_{u,v}$ for
  $u,v \in W^P$.
\end{lemma}
\begin{proof}
  We have $\cO^u \cdot \cO_v = [\cO_{X^u \cap X_v}]$ in $K^T(X)$.
  Since $X^u \cap X_v$ is empty for $u \not\leq v$, and is rational
  with rational singularities for $u \leq v$
  \cite{richardson:intersections, brion:positivity}, we obtain
  \[
  \euler{X}(\cO^u \cdot \cO_v) = \begin{cases}
    1 & \text{if $u \leq v$;}\\
    0 & \text{if $u \not\leq v$.}
  \end{cases}
  \]
  It follows that $\euler{X}(\theta_0(\cO^u), \cO_v)$ is the sum of
  the signs $(-1)^{\ell(w/u)}$ over the set $S$ of all $w \in W^P$ for
  which $u \leq w \leq v$ and $w/u$ is a rook strip.  This sum is
  equal to 1 if $u=v$, and it is empty if $u \not\leq v$.  Assume that
  $u < v$ and let $\al$ be any minimal box of the skew shape
  $I(v) \ssm I(u)$.  Then the map $w \mapsto w s_\al$ is an involution
  of $S$.  Since $\ell(w s_\al) = \ell(w) \pm 1$ for all $w \in S$, we
  have $(-1)^{\ell(w/u)} + (-1)^{\ell(w s_\al/u)} = 0$.  This implies
  that $\euler{X}(\theta_0(\cO^u), \cO_v) = 0$, as required.
\end{proof}

For elements $u, w \in W^P$ such that $u \leq w$ and $w/u$ is a short
rook strip, we define an element $\delta(w/u)$ of the root lattice by
\[
\delta(w/u) = \sum_{\al \in I(w)\ssm I(u)} \delta(\al) \,.
\]
Equivalently, $\delta(w/u)$ is the sum of the (short) simple roots
$\be \in \Delta$ for which the reflection $s_\be$ appears in a reduced
expression for $w/u$.  Define an additional homomorphism of
$\Gamma$-modules $\phi : K^T(X) \to K^T(X)$ by
\[
\phi(\cO^u) = \sum_{w/u \text{ short rook strip}}
(-1)^{\ell(w/u)}\, [\C_{-\delta(w/u)}]\, \cO^w
\]
where the sum is over all $w \in W^P$ for which $u \leq w$ and $w/u$
is a short rook strip.  If the root system $\Phi$ is simply laced,
then $\phi$ is the identity map.  The following result is a special
case of Lenart and Postnikov's Chevalley formula for the equivariant
$K$-theory of arbitrary flag varieties \cite{lenart.postnikov:affine}.

\begin{thm}\label{thm:ktchev}
  Let $X = G/P$ be a cominuscule space.  For any $u \in W^P$ we have
  \[
  J \cdot \cO^u \, = \, J_u\, \theta_0(\phi(\cO^u))
  \]
  in $K^T(X)$.
\end{thm}

\Theorem{ktchev} can be derived from Lenart and Postnikov's more
general result as follows.  Let $\al_1, \al_2, \dots, \al_s$ be any
ordering of the roots in $\cP_X$ that is compatible with the partial
order $\leq$, and let $\be_1, \be_2, \dots, \be_t$ be an ordering of
the short roots of $\cP_X$ that is compatible with $\leq$.  Then one
can show that the sequence $(\al_1,\dots,\al_s, \be_1,\dots,\be_t)$ is
what is called a $\la$-chain for the weight $\la = \omega_\ga$ in
\cite{lenart.postnikov:affine}, after which \Theorem{ktchev} is
equivalent to \cite[Thm.~13.1]{lenart.postnikov:affine}.
In \Section{chevproof} we give an alternative proof of
\Theorem{ktchev} based on geometric considerations.

Notice that since the cominuscule root $\ga$ is long and the root
system $\Phi$ does not have type ${\rm G}_2$, we have
$(\omega_\ga, \al^\vee) = 2$ for any short root $\al \in \cP_X$.
\Lemma{Ju} therefore implies that $J_w = J_u [\C_{-2\delta(w/u)}]$
whenever $w/u$ is a short rook strip.  For this reason we will use the
notation
$\sqrt{J_u J_w} = J_u [\C_{-\delta(w/u)}] = [\C_{u.\omega_\ga -
  \omega_\ga - \delta(w/u)}]$.
The identity of \Theorem{ktchev} can be rewritten as
\[
J \cdot \cO^u \ = \ \sum_{w/u \text{ short rook strip}}
(-1)^{\ell(w/u)} \sqrt{J_u J_w}\ \theta_0(\cO^w) \,.
\]
It therefore follows from \Lemma{th0dual} that \Theorem{ktchev} is
equivalent to the following geometric identity, which is interesting
by itself.

\begin{thm}\label{thm:ktchev2}
  For $u, w \in W^P$ we have
  \[
  \euler{X}(J \cdot \cO^u \cdot \cO_w) \ = \
  \begin{cases}
    (-1)^{\ell(w/u)}\, \sqrt{J_u J_w} &
    \text{if $w/u$ is a short rook strip,} \\
    0 & \text{otherwise.}
  \end{cases}
  \]
\end{thm}

\subsection{Chevalley formula for $\QK_T(X)$}

To state our Chevalley formula for the quantum $K$-theory ring, we
define the homomorphism $\theta_1 : K^T(X) \to K^T(X)$ of
$\Gamma$-modules as follows.  Let $u \in W^P$.  If $u \geq z_1$ then
we set
\[
\theta_1(\cO^u) \ = \ \sum_w \ (-1)^{\ell(w/u(-1))}\, \cO^w
\]
where the sum is over all $w \in W^P$ for which
$u(-1) \leq w \leq 1^\vee(-1) = z_1^\vee$ and $w/u(-1)$ is a rook
strip.  If $u \not\geq z_1$ then set $\theta_1(\cO^u) = 0$.

\begin{lemma}\label{lemma:theta_psi}
  We have $\theta_1 = \psi \, \theta_0$ as $\Gamma$-linear
  endomorphisms of $K^T(X)$.
\end{lemma}
\begin{proof}
  For any $u \in W^P$ we have
  \begin{equation}\label{eqn:thpsi}
    \psi(\theta_0(\cO^u)) = \sum_{w/u \text{ rook strip}}
    (-1)^{\ell(w/u)}\, \cO^{w(-1)} \,.
  \end{equation}
  If $u \geq z_1$ then we have $w(-1) = w/z_1$ and $w(-1)/u(-1) = w/u$
  for each element $w$ in the sum.  Since the map $w \mapsto w(-1)$
  gives a bijection between the terms of \eqn{thpsi} and the sum
  defining $\theta_1(\cO^u)$, the lemma follows in this case.  If
  $u \not\geq z_1$, then let $\al$ be a minimal box in the set
  $I(z_1) \ssm I(u)$.  Then the map $w \mapsto w s_\al$ defines an
  involution of the set of elements $w$ appearing in \eqn{thpsi}.
  Since we have $\ell(w s_\al) = \ell(w) \pm 1$ and the identity
  $w(-1) = (w \cup z_1)/z_1$ implies that $w(-1) = (w s_\al)(-1)$, we
  deduce that $\psi(\theta_0(\cO^u)) = 0$, as required.
\end{proof}

\begin{thm}\label{thm:qktchev}
  Let $X = G/P$ be cominuscule.  For $u \in W^P$ we have
  \[
  J \star \cO^u \ = \ J_u\, \theta_0(\phi(\cO^u)) \, - \,
  J_u\, q\, \theta_1(\phi(\cO^u))
  \]
  in $\QK_T(X)$.
\end{thm}
\begin{proof}
  Since all curves of positive degree in $X$ meet the Schubert divisor
  $X^{s_\ga}$, we have
  $\Gamma_d(X_u, X^{s_\ga}) = \Gamma_d(X_u) = \Gamma_d(X_u,X)$ for all
  $d \geq 1$.  \Theorem{qclas} therefore implies that
  \[
  I_d(\cO_u, J, \cO_w^\vee) \ = \ I_d(\cO_u,1,\cO_w^\vee) -
  I_d(\cO_u,\cO^{s_\ga},\cO_w^\vee) \ = \ 0
  \]
  for all $u, w \in W^P$ and $d \geq 1$.  By linearity we obtain
  $I_d(\sigma, J, \cO_w^\vee) = 0$ for all $\sigma \in K^T(X)$ and
  $d \geq 1$, which is equivalent to the identity
  $\sigma \odot J = \sigma \cdot J$.  We therefore obtain
  \[
  \begin{split}
    J \star \cO^u \
    &= \ (1 - q \psi)(J \odot \cO^u)
    \ = \ (1 - q \psi)(J \cdot \cO^u) \\
    &= \ (1 - q \psi)(J_u\, \theta_0(\phi(\cO^u)))
    \ = \ J_u\, \theta_0(\phi(\cO^u)) - J_u\, q\,
    \psi(\theta_0(\phi(\cO^u)))
  \end{split}
  \]
  so the result follows from \Lemma{theta_psi}.
\end{proof}

\subsection{Structure constants of the Chevalley formula}

Recall that a box $\al$ of a skew shape $I(w) \ssm I(u)$ is {\em
  incomparable\/} if it is both a minimal and maximal box within this
skew shape.  \Theorem{qktchev} can be restated as follows.

\begin{cor}\label{cor:chev_const}
  Let $X = G/P$ be cominuscule.  For $u \in W^P$ we have
  \[
  \cO^{s_\ga} \star \cO^u = \sum_{w \in W^P} N^{w,0}_{s_\ga,u}\, \cO^w
  + \sum_{w \in W^P} N^{w,1}_{s_\ga,u}\, q\, \cO^w
  \]
  where the constants are determined by $N^{u,0}_{s_\ga,u} = 1-J_u$
  and the following rules.

  For $u\neq w$ the constant $N^{w,0}_{s_\ga,u}$ is non-zero if and
  only if $u < w$ and all non-maximal boxes of the skew shape
  $I(w)\ssm I(u)$ are short minimal boxes.  In this~case
  \begin{equation}\label{eqn:Nw0su}
    N^{w,0}_{s_\ga,u} \ = \
    (-1)^{\ell(w/u)-1}\, J_u
    \left( \prod_\al [\C_{-\delta(\al)}] \right)
    \left( \prod_{\al'} (1 + [\C_{-\delta(\al')}]) \right)
  \end{equation}
  where the first product is over all non-maximal boxes $\al$ of
  $I(w) \ssm I(u)$, and the second product is over all short
  incomparable boxes $\al'$ of $I(w)\ssm I(u)$.

  The constant $N^{w,1}_{s_\ga,u}$ is non-zero if and only if
  $u(-1) \leq w \leq z_1^\vee$, all non-maximal boxes of the skew
  shape $I(w(1))\ssm I(u)$ are short minimal boxes, and all roots of
  $I(z_1)\ssm I(u)$ are short minimal boxes of $I(w(1))\ssm I(u)$.  In
  this case we have
  \begin{equation}\label{eqn:Nw1su}
    N^{w,1}_{s_\ga,u} \ = \
    (-1)^{\ell(w(1)/u)}\, J_u
    \left( \prod_\al [\C_{-\delta(\al)}] \right)
    \left( \prod_{\al'} (1 + [\C_{-\delta(\al')}]) \right)
  \end{equation}
  where the first product is over all boxes $\al$ of
  $I(w(1))\ssm I(u)$ that are either non-maximal or belong to
  $I(z_1)$, and the second product is over all short incomparable
  boxes $\al'$ of $I(w(1))\ssm I(u)$ for which $\al' \notin I(z_1)$.
\end{cor}
\begin{proof}
  \Theorem{qktchev} implies that the coefficient of $q\,\cO^w$ in
  $J \star \cO^u$ is given by
  \begin{equation}\label{eqn:Nw1su_sum}
    - N^{w,1}_{s_\ga,u} \ = \ (-1)^{\ell(w)+\ell(z_1)+1-\ell(u)}\, J_u\,
    \sum_v\, [\C_{-\delta(v/u)}]
  \end{equation}
  where the sum is over all $v \in W^P$ for which $\cO^v$ appears in
  $\phi(\cO^u)$ and $\cO^w$ appears in $\theta_1(\cO^v)$.  The first
  condition says that $u \leq v$ and $v/u$ is a short rook strip, and
  by \Lemma{comin_nbhd} the second condition holds if and only if
  $w \leq z_1^\vee$, $z_1 \leq v \leq w(1)$, and $w(1)/v$ is an
  arbitrary rook strip.  Such elements $v$ exist if and only if the
  conditions for $N^{w,1}_{s_\ga,u}\neq 0$ in the corollary are
  satisfied, and in this case $v$ appears in the sum if and only if
  $I(v)$ is the union of $I(u)$, the set of all roots $\al$ in the
  first product of \eqn{Nw1su}, and an arbitrary subset of the roots
  $\al'$ in the second product.  The identity \eqn{Nw1su} now follows
  because the last two products in this identity expand to the sum in
  equation \eqn{Nw1su_sum}.  The proof of \eqn{Nw0su} is similar and
  left to the reader.
\end{proof}

\begin{remark}\label{remark:slaced}
  Assume that the root system $R$ is simply laced.  Then there are no
  short boxes, and the statement of \Corollary{chev_const} simplifies
  considerably.  Any coefficient $N^{w,0}_{s_\ga,u}$ with $w \neq u$
  is non-zero if and only if $u < w$ and $w/u$ is a rook strip, in
  which case
  \[
  N^{w,0}_{s_\ga,u} \ = \ (-1)^{\ell(w/u)-1}\, J_u \,.
  \]
  Similarly, $N^{w,1}_{s_\gamma,u}$ is non-zero if and only if
  $u(-1) \leq w \leq z_1^\vee$, $z_1 \leq u$, and $w(1)/u$ is a rook
  strip.  In this case we have
  \[
  N^{w,1}_{s_\ga,u} \ = \ (-1)^{\ell(w(1)/u)}\, J_u \,.
  \]
\end{remark}

We proceed to give two examples of the Chevalley formula in types A
and C.  More examples can be obtained using the {\em Equivariant
  Schubert Calculator} available at
\url{http://sites.math.rutgers.edu/~asbuch/equivcalc/}.

\begin{example}\label{example:gr37}
  Let $X = \Gr(3,7) = \{ V \subset \C^7 \mid \dim(V)=3 \}$ be the
  Grassmann variety of $3$-planes in $\C^7$.  This is a homogeneous
  space for the group $G = \SL(\C^7)$ of Lie type ${\rm A}_6$.  Let
  $T \subset G$ be the maximal torus of diagonal matrices and let
  $B \subset G$ be the Borel subgroup of upper triangular matrices.
  For $1 \leq i \leq 7$ we let $\vep_i : T \to \C^*$ be the character
  that sends a diagonal matrix to its $i$-th entry.  Then the set of
  simple roots of $G$ is
  $\Delta = \{ \vep_i - \vep_{i+1} \mid 1 \leq i \leq 6 \}$, and all
  these simple roots are cominuscule.  We can identify the homogeneous
  space $G/P$ corresponding to $\ga = \vep_3-\vep_4$ with the
  Grassmannian $X$ by the map $g.P \mapsto g.E_3$, where
  $E_3 = \Span\{e_1,e_2,e_3\}$ is the span of the first three standard
  basis vectors in $\C^7$.  The partially ordered set $\cP_X$ consists
  of the boxes of the following $3 \times 4$ rectangle.
  \[
  \def\sroot#1#2{\vep_{#1}\!\!-\!\vep_{#2}}%
  \tableau{29}{
  {\ga}&{\sroot35}&{\sroot36}&{\sroot37}\\
  {\sroot24}&{\sroot25}&{\sroot26}&{\sroot27}\\
  {\sroot14}&{\sroot15}&{\sroot16}&{\sroot17}}
  \]
  The partial order is given by $\vep_i-\vep_j \leq \vep_k-\vep_l$ if
  and only if $k \leq i$ and $j \leq l$, or equivalently, the box
  $\vep_i-\vep_j$ is north-west of $\vep_k-\vep_l$.  The order ideals
  of $\cP_X$ therefore correspond to Young diagrams $\la$ contained in
  the upper-left corner of the rectangle.  Notice that
  $\delta(\vep_i-\vep_j) = \vep_{i+j-4}-\vep_{i+j-3}$.  Since we have
  $I(z_1) = \tableau{5}{{}&{}&{}&{}\\{}\\{}}$, the map
  $w \mapsto w(-1)$ corresponds to the map of Young diagrams that
  removes the first row and the first column, and moves the remaining
  boxes one step north-west.  Similarly, the map $w \mapsto w(1)$
  corresponds to the map of Young diagrams that moves all boxes one
  step south-east, discards any boxes that leave the $3 \times 4$
  rectangle, and adds the boxes of $I(z_1)$ to the result.  If
  $w \le z_1^\vee$ then no boxes need to be discarded; equivalently we
  have $w(1)(-1) = w$.

  Consider the diagram $\mu = \tableau{5}{{}&{}&{}&{}\\{}&{}&{}\\{}}$,
  and let $w_\mu \in W^P$ be the corresponding Weyl group element.
  Using \Lemma{Ju} we obtain
  $J_{w_\mu} = [\C_{\vep_7+\vep_5-\vep_3-\vep_1}]$.  We will abuse
  notation and denote the Schubert class $\cO^{w_\la}$ simply by
  $\la$.  Since $\phi(\cO^{w_\mu}) = \cO^{w_\mu}$ and
  $I(w_\mu(-1)) = \tableau{5}{{}&{}}\,$, it follows from
  \Theorem{qktchev} that
  \[
  J \star\, \tableau{5}{{}&{}&{}&{}\\{}&{}&{}\\{}} \ = \
  J_{w_\mu} \left(\,
    \tableau{5}{{}&{}&{}&{}\\{}&{}&{}\\{}}
    - \tableau{5}{{}&{}&{}&{}\\{}&{}&{}&{}\\{}}
    - \tableau{5}{{}&{}&{}&{}\\{}&{}&{}\\{}&{}}
    + \tableau{5}{{}&{}&{}&{}\\{}&{}&{}&{}\\{}&{}}
    - q\, \tableau{5}{{}&{}}
    + q\, \tableau{5}{{}&{}&{}}
    + q\, \tableau{5}{{}&{}\\{}}
    - q\, \tableau{5}{{}&{}&{}\\{}}
    \,\right)
  \,.
  \]
  Equivalently, the expansion of the product
  $\cO^{s_\ga} \star \cO^{w_\mu} = \cO^{w_\mu} - J \star \cO^{w_\mu}$
  can be obtained using \Remark{slaced}.
\end{example}

\begin{example}[Lagrangian Grassmannian]\label{example:LG}
  Set $E = \C^{2n}$, let $\{e_1,e_2,\dots,e_{2n}\}$ be the standard
  basis, and define a symplectic form on $E$ by
  $(e_i,e_j) = -(e_j,e_i) = \delta_{i+j,2n+1}$ for
  $1 \leq i \leq j \leq 2n$.  The Lagrangian Grassmannian of maximal
  isotropic subspaces of $E$ is the variety
  \[
  X = \LG(n,E) = \{ V \subset E \mid \dim(V) = n \text{ and }
  (V,V)=0 \} \,.
  \]
  This is a homogeneous space for the symplectic group
  $G = \Sp(E) = \{ g \in \GL(E) \mid (g.x,g.y) = (x,y) ~\forall x,y
  \in E \}$
  of Lie type ${\rm C}_n$.  Let $T \subset G$ be the maximal torus of
  diagonal matrices, and let $B \subset G$ be the Borel subgroup of
  upper triangular matrices.  For $1 \leq i \leq 2n$ we let
  $\vep_i : T \to \C^*$ be the character that sends a diagonal matrix
  to its $i$-th entry.  Since all elements of $T$ have the form
  $\diag(t_1,\dots,t_n, t_n^{-1},\dots,t_1^{-1})$, we have
  $\vep_{2n+1-i} = -\vep_i$ for each $i$.  The set of simple roots of
  $G$ is
  $\Delta = \{ \vep_1-\vep_2, \vep_2-\vep_3, \dots, \vep_{n-1}-\vep_n,
  2\vep_n \}$.
  This set contains a single cominuscule simple root $\ga = 2 \vep_n$.
  The corresponding homogeneous space $G/P$ can be identified with $X$
  by the map $g.P \mapsto g.E_n$, where
  $E_n = \Span\{e_1,\dots,e_n\}$.  The partially ordered set $\cP_X$
  consists of the boxes in the staircase diagram
  \[
  \def\sroot#1#2{\vep_{#1}\!\!+\!\vep_{#2}}%
  \tableau{29}{
    \ga&\hdots&\sroot3n&\sroot2n&\sroot1n\\
    &\ddots&\vdots&\vdots&\vdots\\
    &&2\vep_3&\sroot23&\sroot13\\
    &&&2\vep_2&\sroot12\\
    &&&&2\vep_1
  } \hspace{5pt}\mbox{}
  \]
  where boxes increase from north-west to south-east as in
  \Example{gr37}.  The long roots of $\cP_X$ are the $n$ diagonal
  boxes.  Each element $u \in W^P$ corresponds to a shifted Young
  diagram $I(u)$ of boxes in the upper-left corner of the staircase
  diagram.  Since $I(z_1)$ is the top row of boxes in $\cP_X$, the map
  $w \mapsto w(-1)$ corresponds to the map of shifted Young diagrams
  that removes the top row and moves the remaining boxes one step
  north-west.  The map $w \mapsto w(1)$ corresponds to the map of
  shifted Young diagrams that moves all boxes one step south-east,
  discards the boxes that leave the staircase diagram $\cP_X$, and
  adds the boxes of $I(z_1)$ to the result.  If $I(w)$ does not
  contain boxes in the right-most column of $\cP_X$, then no boxes are
  discarded.

  To illustrate the Chevalley formula, we consider the case $n=5$ and
  $\mu = \tableau{5}{{}&{}&{}&{}\\&{}&{}}\,$.  If we denote each
  Schubert class $\cO^{w_\la}$ by $\la$ as above, then we have
  \[
  \phi(\,\tableau{5}{{}&{}&{}&{}\\&{}&{}}\,) \ = \
  \tableau{5}{{}&{}&{}&{}\\&{}&{}} \, - \, [\C_{\vep_2 - \vep_1}]\,
  \tableau{5}{{}&{}&{}&{}&{}\\&{}&{}} \, - \, [ \C_{\vep_4 - \vep_3}]\,
  \tableau{5}{{}&{}&{}&{}\\&{}&{}&{}} \, + \, [\C_{\vep_4 - \vep_3 +
    \vep_2 - \vep_1}]\,\tableau{5}{{}&{}&{}&{}&{}\\&{}&{}&{}} \,.
  \]
  \Lemma{Ju} gives $J_{w_\mu} = [\C_{-2(\vep_2+\vep_4)}]$, and
  \Theorem{qktchev} shows that the quantum terms of the product
  $J \star \cO^{w_\mu}$ are given by
  $- J_{w_\mu} q\, \theta_1(\phi(\cO^{w_\mu}))$, which expands to
  \[
  q [\C_{-\vep_1-\vep_2-2\vep_4}]\, (
  \tableau{5}{{}&{}} -
  \tableau{5}{{}&{}&{}} -
  \tableau{5}{{}&{}\\&{}}+
  \tableau{5}{{}&{}&{}\\&{}}
  )
  - q [\C_{-\vep_1-\vep_2-\vep_3-\vep_4}] (
  \tableau{5}{{}&{}&{}} -
  \tableau{5}{{}&{}&{}&{}} -
  \tableau{5}{{}&{}&{}\\&{}} +
  \tableau{5}{{}&{}&{}&{}\\&{}}
  ) \,.
  \]
  Alternatively we can use \Corollary{chev_const} to compute the
  coefficients $N^{w_\nu,1}_{s_\ga,w_\mu}$ in the expansion of
  $\cO^{s_\ga} \star \cO^{w_\mu}$.  For example, for
  $\nu = \tableau{5}{{}&{}&{}\\&{}}$ we obtain
  $I(w_\nu(1)) = \tableau{5}{{}&{}&{}&{}&{}\\&{}&{}&{}\\&&{}}$, and
  noting that $I(z_1) \ssm \mu$ consists of the single short box
  $\vep_1+\vep_5$, we obtain
  \[
  N^{w_\nu,1}_{s_\ga,w_\mu} \ = \ (-1)^3\, J_u\, [\C_{\vep_2-\vep_1}]\,
  (1 + [\C_{\vep_4-\vep_3}]) \,.
  \]
\end{example}

\section{Geometric proof of the Chevalley formula}
\label{sec:chevproof}

\subsection{Minuscule varieties}\label{sec:minuscule}

This section contains a geometric proof of \Theorem{ktchev}.  As in
the previous section we assume that $X = G/P$ is a cominuscule
variety, and $\ga \in \Delta$ is the corresponding cominuscule simple
root.  While our proof is very simple when the root system $\Phi$ is
simply laced, we need to resort to specialized arguments in the
non-simply laced cases.

We may assume without loss of generality that $G$ is simply connected,
so that any integral weight of the root system $\Phi$ is represented
by a character of $T$.  The fundamental weight $\omega_\ga$ extends to
a character $\omega_\ga : P \to \C^*$.  Let
$L = G \times^P \C_{-\omega_\ga}$ be the associated line bundle on
$X$.  The total space of this bundle consists of all pairs $[g,z]$
with $g \in G$ and $z \in \C$, subject to the relation
$[gp,\omega_\ga(p)z] = [g,z]$ for all $p \in P$.  We regard $L$ as a
$G$-equivariant line bundle with action defined by
$g'.[g,z] = [g'g,z]$.  In particular $G$ acts on the vector space of
global sections $H^0(X,L)$.  According to the Borel-Weil theorem,
$H^0(X,L)^*$ is an irreducible representation of $G$ with highest
weight $\omega_\ga$.  It follows that there exists a $B^-$-stable
section $\sigma \in H^0(X,L)$ of weight $-\omega_\ga$.  Given any
element $u \in W$, we will misuse notation and write
$u.\sigma = \Dot{u}.\sigma$, where $\Dot{u} \in N_G(T)$ is a
representative of $u$.  The section $u.\sigma \in H^0(X,L)$ is well
defined only up to scalar.

\begin{prop}\label{prop:boundary}
  Let $u \in W^P$.  The support of $Z(u.\sigma) \cap X^u$ is the
  boundary $\partial X^u$.  If all minimal boxes of $\cP_X \ssm I(u)$
  are long roots, then $Z(u.\sigma) \cap X^u = \partial X^u$ as
  (reduced) subschemes of $X^u$.  In particular, if $X$ is a minuscule
  variety, then the boundary $\partial X^u$ is a Cartier divisor in
  $X^u$.
\end{prop}
\begin{proof}
  Since $\sigma$ is a $B^-$-stable section of $L$, it follows that the
  zero section $Z(\sigma)$ is a $B^-$-stable divisor, so we have
  $Z(\sigma) = X^{s_\ga}$ as a set.  It follows that
  $1.P \notin Z(\sigma)$, and hence $u.P \notin Z(u.\sigma)$, so
  \Lemma{tstable} implies that
  $Z(u.\sigma) \cap X^u \subset \partial X^u$.  Let
  $\tau \in H^0(X^u, L|_{X^u})$ denote the section obtained by
  restricting $u.\sigma$ to $X^u$.  The cycle $[Z(\tau)]$ is a sum of
  prime divisors of $X^u$ contained in the boundary $\partial X^u$.
  These prime divisors are exactly the Schubert varieties $X^w$ for
  which $w$ covers $u$ in the Bruhat order on $W^P$, i.e.\ $u \leq w$
  and $\ell(w)=\ell(u)+1$.  Furthermore, the classes $[X^w]$ of these
  prime divisors are linearly independent in $H_*(X^u;\Z)$ as they map
  to linearly independent classes in $H_*(X;\Z)$.  For each
  $w \in W^P$ that covers $u$ we have $I(w) \ssm I(u) = \{\al\}$ for
  some root $\al \in \cP_X$ such that $w = u s_\al$, and it follows
  from the classical Chevalley formula \cite{chevalley:decompositions}
  that the coefficient of $[X^w]$ in $[Z(\tau)]$ is equal to
  $(\omega_\ga,\al^\vee)$.  Since the cominuscule root $\ga$ is long,
  and since the coefficient of $\ga$ in $\al$ is one, it follows that
  $(\omega_\ga,\al^\vee)$ is equal to 1 if $\al$ is long and equal to
  2 if $\al$ is short.  In particular, $\tau$ vanishes along each
  Schubert divisor contained in $\partial X^u$, so we have
  $Z(u.\sigma) \cap X^u = Z(\tau) = \partial X^u$ as sets.  If all
  minimal boxes of $\cP_X \ssm I(u)$ are long roots, then $\tau$
  vanishes to the first order along each prime divisor
  $X^w \subset X^u$.  In other words the ideal of $Z(\tau)$ agrees
  with the maximal ideal in each discrete valuation ring
  $\cO_{X^w,X^u}$.  By using that $X^u$ is normal and the fact that
  any normal ring is the intersection of its localizations at prime
  ideals of codimension one, we deduce that $Z(\tau)$ is a reduced
  subscheme of $X^u$.  This completes the proof.
\end{proof}

It follows from \Proposition{boundary} that $X^{s_\ga} = Z(\sigma)$ as
a subscheme of $X$.  Since $\sigma$ has weight $-\omega_\ga$ in
$H^0(X,L)$, we may consider $\sigma$ as a morphism
$\sigma : X \times \C_{-\omega_\ga} \to L$ of $T$-equivariant vector
bundles on $X$.  We obtain a $T$-equivariant exact sequence
\[
0 \to L^\vee \otimes \C_{-\omega_\ga} \to \cO_X \to \cO_{X^{s_\ga}}
\to 0
\]
which reveals the identity
$J = 1 - \cO^{s_\ga} = [L^\vee \otimes \C_{-\omega_\ga}]$ in $K^T(X)$.

\begin{proof}[Proof of \Lemma{Ju}]
  The first identity of the lemma follows since
  $L|_{u.P} \cong \C_{-u.\omega_\ga}$ as a representation of $T$.  The
  second identity is clear if $u=1$.  If $u \neq 1$ then choose
  $u' \in W^P$ such that $u' \leq u$ and $\ell(u') = \ell(u)-1$.  Then
  $I(u) = I(u') \cup \{\al\}$ for some root $\al \in \cP_X$ such that
  $u = u' s_\al$, and we have $\delta(\al) = u'.\al$.  We obtain
  \[
  \omega_\ga - u.\omega_\ga =
  \omega_\ga - u'.(\omega_\ga - (\omega_\ga,\al^\vee)\,\al) =
  (\omega_\ga - u'.\omega_\ga) + (\omega_\ga,\al^\vee)\,\delta(\al)
  \,,
  \]
  so the required identity follows by induction on $\ell(u)$.
\end{proof}

The following result implies \Theorem{ktchev} for all minuscule
varieties.

\begin{prop}\label{prop:noshort}
  \Theorem{ktchev2} is true whenever all minimal boxes of
  $\cP_X \ssm I(u)$ are long roots.
\end{prop}
\begin{proof}
  In the situation of the proposition, \Theorem{ktchev2} states that
  $J_u^{-1} J\, \cO^u$ is dual to $\cO_u$.  \Proposition{boundary}
  implies that
  $I_{\partial X^u} \cong (L^\vee \otimes \C_{-u.\omega_\ga})|_{X^u}$,
  so
  $J_u^{-1} J\, \cO^u = [L^\vee \otimes \C_{-u.\omega_\ga}]\, \cO^u =
  [I_{\partial X^u}]$
  is dual to $\cO_u$ by Brion's identity \eqn{brion_dual}.
\end{proof}

\subsection{Lagrangian Grassmannians}\label{sec:lagrange}

We next prove \Theorem{ktchev2} for Lagrangian Grassmannians
$X = \LG(n,E)$, using the notation introduced in \Example{LG}.

Let $\bP E$ be the projective space of lines through the origin in
$E$, set
$\SF(1,n;E) = \{ (L,V) \in \bP E \times X \mid L \subset V \}$, and
let $\pi : \SF(1,n;E) \to X$ and $\varphi : \SF(1,n;E) \to \bP E$ be
the projections.  Both of these projections are $T$-equivariant maps.
Set $E^n = \Span\{e_{n+1}, e_{n+2}, \dots, e_{2n}\}$ and notice that
the opposite Schubert divisor in $X$ is defined by
$X^{s_\ga} = \{ V \in X \mid V \cap E^n \neq 0 \} =
\pi(\varphi^{-1}(\bP E^n))$.
Since $\pi$ maps $\varphi^{-1}(\bP E^n)$ birationally onto $X^{s_\ga}$
we obtain $\cO^{s_\ga} = \pi_* \varphi^*[\cO_{\bP E^n}]$ in $K^T(X)$.

Let $u, w \in W^P$ be elements satisfying $u \leq w$.  Then we have
$\cO^u \cdot \cO_w = [\cO_{X^u \cap X_w}]$ and
$\euler{X}(\cO^u \cdot \cO_w) = 1$.  Set
$Z^u_w = \varphi(\pi^{-1}(X^u \cap X_w)) \subset \bP E$.  It follows
from \cite[Lemma~5.1]{buch.ravikumar:pieri} or
\cite[Thm.~4.5]{knutson.lam.ea:projections} that
$\varphi_* \pi^*(\cO^u \cdot \cO_w) = [\cO_{Z^u_w}]$ in $K^T(\bP E)$.
By the projection formula we therefore obtain
\[
\euler{X}(J \cdot \cO^u \cdot \cO_w) =
\euler{X}(\cO^u \cdot \cO_w) - \euler{X}(\cO^{s_\ga} \cdot \cO^u \cdot
\cO_w) =
1 - \euler{\bP E}([\cO_{\bP E^n}] \cdot [\cO_{Z^u_w}]) \,.
\]

It is known from \cite[Lemma~5.2]{buch.ravikumar:pieri} that $Z^u_w$
is a complete intersection in $\bP E$ defined by linear and quadratic
equations.  We need the following explicit description of these
equations, which is a consequence of the proof given in
\cite{buch.ravikumar:pieri}.  Notice that the south-east boundary of
the shape $I(u)$ is a path of $n$ line segments from the upper-right
corner of the staircase diagram $\cP_X$ to its diagonal (see
\Example{projrich}).  The $k$-th step of this path will be called the
$k$-th {\em boundary segment\/} of $u$.  Let $\C[x_1,\dots,x_{2n}]$ be
the projective coordinate ring of $\bP E$, with $x_i$ dual to the
basis element $e_i \in E$.  The linear equations of $Z^u_w$ correspond
to the boundary segments that $u$ and $w$ have in common.  Assume that
the $k$-th boundary segment of $u$ is also a boundary segment of $w$.
If this boundary segment is horizontal, then it contributes the linear
equation $x_{2n+1-k} = 0$, and if it is vertical, then it contributes
the equation $x_k = 0$.  Two boxes in the staircase diagram are {\em
  connected\/} if they share a side.  The quadratic equations of
$Z^u_w$ correspond to connected components of the skew shape
$I(w) \ssm I(u)$ which do not contain any (long) diagonal boxes of
$\cP_X$.  These components will be called {\em short components\/} of
$w/u$.  Assume that the north-west side of a short component consists
of the boundary segments of $u$ numbered from $i$ to $j$.  Then the
component contributes the quadratic equation
$f_{ij} = x_i x_{2n+1-i} + x_{i+1} x_{2n-i} + \dots + x_j x_{2n+1-j} =
0$.

\begin{example}\label{example:projrich}
  Let $X = \LG(7,14)$ be the Lagrangian Grassmannian of type
  ${\rm C}_7$ and let $u \leq w$ be the elements of $W^P$
  corresponding to the marked boundaries.
  \[
  \tableau{12}{
    {}&{}&{}&{}&{}&[TTTVVVb]&[TTTvb] \\
    &[a]&{}&{}&{}&[tVVVBBB]&[a] \\
    &&{}&{}&[ltRRRBBB]&[a]&{} \\
    &&&[TTTRRRlb]&[a]&{}&{} \\
    &&&&{}&{}&{} \\
    &&&&&{}&{} \\
    &&&&&&{}
  }
  \]
  Then $Z^u_w \subset \bP^{13}$ is defined by
  $x_5 = x_9 = x_{14} = x_2x_{13}+x_3x_{12}+x_4x_{11}=0$.
\end{example}

Let $\cO_{\bP E}(-1) \subset \bP E \times E$ denote the tautological
subbundle, with the equivariant structure defined by the action of $G$
on $E$, and set $h = [\cO_{\bP E}(-1)] \in K^T(\bP E)$.  Since $T$
acts on $x_i$ with weight $-\vep_i$ and on $f_{ij}$ with weight
$\vep_i+\vep_{2n+1-i} = 0$, we have
$[\cO_{Z(x_i)}] = 1 - [\C_{-\vep_i}] h$ and
$[\cO_{Z(f_{ij})}] = 1 - h^2$ in $K^T(\bP E)$.  Let $l$ be the number
of boundary segments shared by $u$ and $w$, let
$\mu_1, \mu_2, \dots, \mu_l$ be the weights of the corresponding
linear equations, and let $q$ be the number of short components of
$w/u$.  We then have
\[
[\cO_{\bP E^n}] \cdot [\cO_{Z^u_w}] =
\left( \prod_{j=1}^n (1 - [\C_{-\vep_j}] h) \right)\!
\left( \prod_{i=1}^l (1 - [\C_{\mu_i}] h) \right) (1 - h^2)^q
\ = \
1 + \sum_{k=1}^{n+l+2q} a_k\, h^k
\]
for some classes $a_k \in \Gamma$.  The leading coefficient is
$a_{n+l+2q} = (-1)^{n+l+q} [\C_{\nu(u,w)}]$, where
$\nu(u,w) = \sum_{i=1}^l \mu_i - \sum_{j=1}^n \vep_j$.

\begin{lemma}
  Let $F$ be any representation of $T$ of dimension $d$, and consider
  the sheaf Euler characteristic map
  $\euler{\bP F} : K^T(\bP F) \to \Gamma$.  Then we have
  \[
  \euler{\bP F}([\cO_{\bP F}(-k)]) \ = \
  \begin{cases}
    1 & \text{if $k=0$,} \\
    0 & \text{if $0 < k < d$, and} \\
    (-1)^{d-1}\, [\bigwedge^d F] & \text{if $k=d$.}
  \end{cases}
  \]
\end{lemma}
\begin{proof}
  This follows from Serre duality
  \cite[III.7.7]{hartshorne:algebraic*1} because the canonical
  sheaf on $\bP F$ is given by
  $\omega_{\bP F} = (\bigwedge^d F^*) \otimes \cO_{\bP F}(-d)$ as a
  $T$-equivariant vector bundle.
\end{proof}

Since each short component of $w/u$ occupies at least two of the
boundary segments of $u$, we have $l + 2q \leq n$, with equality if
and only if $w/u$ is a short rook strip, and in this case we have
$\ell(w/u) \equiv n+l+q$ (mod 2).  Since
$[\bigwedge^{2n} E] = 1 \in \Gamma$ we obtain
\[
\euler{X}(J \cdot \cO^u \cdot \cO_w) \ = \
\begin{cases}
  (-1)^{\ell(w/u)}\, [\C_{\nu(u,w)}] &
  \text{if $w/u$ is a short rook strip;} \\
  0 & \text{otherwise.}
\end{cases}
\]
As special cases we have
$J_u = \euler{X}(J \cdot \cO^u \cdot \cO_u) = [\C_{\nu(u,u)}]$ and
$J_w = [\C_{\nu(w,w)}]$.  To finish the proof of \Theorem{ktchev2} it
is therefore enough to show that $2\,\nu(u,w) = \nu(u,u) + \nu(w,w)$
whenever $w/u$ is a short rook strip.  To see this, let
$\vep_i+\vep_j \in I(w)\ssm I(u)$, and set $k = n+i-j$.  Then the
$k$-th boundary segment of $u$ is horizontal while the $k$-th boundary
segment of $w$ is vertical, so $Z^u_u$ satisfies the equation
$x_{2n+1-k}=0$ while $Z^w_w$ satisfies the equation $x_k=0$.  It
follows that the weights $-\vep_{2n+1-k}$ and $-\vep_k$ of these
equations cancel out in the sum $\nu(u,u) + \nu(w,w)$, and the same
happens for the weights corresponding to the $k+1$-st boundary
segments.  This shows that $2\nu(u,w) = \nu(u,u) + \nu(w,w)$ and
completes the proof.

\subsection{Odd quadric hypersurfaces}\label{sec:oddquadric}

We finally prove \Theorem{ktchev} for odd quadric hypersurfaces, which
is the last case needed to complete the proof of the Chevalley
formula.  Set $E = \C^{2n+1}$, with basis
$\{e_1, e_2, \dots, e_{2n+1}\}$, and define an orthogonal form on $E$
by $(e_i,e_j) = \delta_{i+j,2n+2}$.  The corresponding symmetry group
$G = \SO(E)$ has Lie type ${\rm B}_n$.  Let $T \subset G$ be the
maximal torus of of diagonal matrices, and let $B \subset G$ be the
Borel subgroup of upper triangular matrices.  For $1 \leq i \leq 2n+1$
we let $\vep_i : T \to \C^*$ be the character that maps a diagonal
matrix to its $i$-th entry.  We then have $\vep_{n+1}=0$ and
$\vep_i + \vep_{2n+2-i} = 0$ for $1 \leq i \leq 2n+1$.  The set of
simple roots is
$\Delta = \{\vep_1-\vep_2, \vep_2-\vep_3, \dots, \vep_{n-1}-\vep_n,
\vep_n\}$,
the unique cominuscule simple root is $\ga = \vep_1-\vep_2$, and
$\omega_\ga = \vep_1$.  Let $\C[x_1,\dots,x_{2n+1}]$ be the coordinate
ring of $E$ and set
$q = x_1 x_{2n+1} + x_2 x_{2n} + \dots + x_n x_{n+2} + \frac{1}{2}
x_{n+1}^2$.
The cominuscule variety $G/P$ corresponding to $\ga$ can be identified
with the quadric hypersurface $X = Z(q) \subset \bP E$ of dimension
$2n-1$.  The set $\cP_X = \{ \vep_1-\vep_i \mid 2 \leq i \leq 2n \}$
is represented by the following diagram where the boxes increase from
left to right.
\[
\def\sroot#1#2{\vep_1\!\!#1\!\vep_{#2}}%
\tableau{29}{{\ga}&{\sroot-3}&{\hdots}&{\sroot-n}&
{\vep_1}&{\sroot+n}&{\hdots}&{\sroot+3}&{\sroot+2}}
\]

The $T$-fixed points of $X$ are the points
$\langle e_k \rangle = \C e_k$ for $1 \leq k \leq 2n+1$ and
$k \neq n+1$.  For $k \in [0,2n-1]$ we let $X^k$ denote the
$B^-$-stable Schubert variety of codimension $k$ in $X$.  More
precisely we have
$X^k = \ov{B^-.\langle e_{k+1} \rangle} = Z(x_1,\dots,x_k,q)$ for
$0 \leq k \leq n-1$, and
$X^k = \ov{B^-.\langle e_{k+2} \rangle} = Z(x_1,\dots,x_{k+1})$ for
$n \leq k \leq 2n-1$.  Set $\cO^k = [\cO_{X^k}]$ and let
$J_k \in K^T(\pt)$ be the restriction of $J$ to the $T$-fixed point
defining $X^k$.  Since $\vep_1$ is the only short box in $\cP_X$, it
follows from \Proposition{noshort} that the Chevalley formula
$J \cdot \cO^k = J_k\, \theta_0(\phi(\cO^k))$ holds whenever
$k \neq n-1$.

Let $\iota : X \to \bP E$ be the inclusion and set
$\wh J = [\C_{-\vep_1}\otimes \cO_{\bP E}(-1)] \in K^T(\bP E)$.  Then
$\iota^*(\wh J) = J$ in $K^T(X)$.  Since the pushforward map
$\iota_* : K^T(X) \to K^T(\bP E)$ is injective, it is enough to show
that the identity
$\wh J \cdot \cO^{n-1} = J_{n-1} \ \theta_0(\phi(\cO^{n-1}))$ holds in
$K^T(\bP E)$.  Since $\vep_n$ is the only short simple root, we deduce
that $\delta(\vep_1) = \vep_n$ and
$\phi(\cO^{n-1}) = \cO^{n-1} - [\C_{-\vep_n}]\,\cO^n$.  Using that
$[\cO_{Z(x_i)}] = 1 - [\C_{-\vep_i}]\, [\cO_{\bP E}(-1)]$ in
$K^T(\bP E)$ for each $i$ we also have
\[
\begin{split}
  [\cO_{Z(q)}]
  &= 1 - [\cO_{\bP E}(-2)]
  = 1 - [\cO_{\bP E}(-1)] + [\cO_{\bP E}(-1)] \cdot [\cO_{Z(x_{n+1})}] \\
  &= [\cO_{Z(x_{n+1})}] + [\C_{-\vep_n}]\, (1 - [\cO_{Z(x_{n+2})}]) \cdot
  [\cO_{Z(x_{n+1})}] \\
  &= [\cO_{Z(x_{n+1})}] + [\C_{-\vep_n}]\, [\cO_{Z(x_{n+1})}] -
  [\C_{-\vep_n}]\, [\cO_{Z(x_{n+1},x_{n+2})}] \,.
\end{split}
\]
Since $X^{n-1} = Z(x_1,\dots,x_{n-1}) \cap Z(q)$ and
$J_{n-1} = \wh J|_{\langle e_n \rangle} = [\C_{\vep_n-\vep_1}]$, we
obtain
\begin{multline*}
  \wh J \cdot \cO^{n-1}
  = J_{n-1} (1 - [\cO_{Z(x_n)}]) \cdot \cO^{n-1}
  = J_{n-1} \cO^{n-1} - J_{n-1} [\cO_{Z(x_1,\dots,x_n)}] \cdot [\cO_{Z(q)}] \\
  = J_{n-1} \cO^{n-1} - J_{n-1} (\cO^n + [\C_{-\vep_n}] \cO^n -
  [\C_{-\vep_n}] \cO^{n+1})
  = J_{n-1}\, \theta_0(\phi(\cO^{n-1})) \,.
\end{multline*}
This completes the proof of \Theorem{ktchev}.

\section{Solutions to Molev-Sagan Equations}
\label{sec:ms}

\subsection{Recursive identities}

In this section we let $X = G/P$ be an arbitrary flag manifold and
show that the $T$-equivariant quantum $K$-theory ring $\QK_T(X)$,
together with its Schubert structure constants and the underlying
(three point, genus zero) Gromov-Witten invariants of $X$, are
uniquely determined by the products that involve divisor classes.
Here $T$ denotes a maximal torus in $G$.

It was observed by Molev and Sagan
\cite{molev.sagan:littlewood-richardson} that the multiplicative
structure constants of factorial Schur polynomials satisfy recursive
identities that we will call {\em Molev-Sagan equations}.  Related
properties of factorial Schur polynomials had previously been studied
by Okounkov and Olshanski \cite{okounkov:quantum,
  okounkov.olshanski:shifted*2}.  Knutson and Tao used the Molev-Sagan
equations to study the equivariant Schubert structure constants of
Grassmannians \cite{knutson.tao:puzzles}.  In this context the
identities follow from the Chevalley formula together with the fact
that the equivariant cohomology ring is associative.  It was shown in
\cite{mihalcea:equivariant*1} that the same method can be applied to
arbitrary flag varieties.  Molev-Sagan type equations have by now been
used to prove Littlewood-Richardson rules in several papers
\cite{molev.sagan:littlewood-richardson, knutson.tao:puzzles,
  buch:mutations, pechenik.yong:equivariant}.

The Molev-Sagan equations provide a triangular system of identities
that reduce arbitrary equivariant Schubert structure constants
$C^w_{u,v}$ to the special constants of the form $C^w_{w,w}$, which
are given by a simple formula of Kostant and Kumar
\cite{kostant.kumar:nil*1}.  The Chevalley formula of Mihalcea
\cite{mihalcea:equivariant*1} for the equivariant quantum cohomology
ring $\QH_T(X)$ can be used in a similar way to produce equations
satisfied by the equivariant Gromov-Witten invariants of $X$, but in
this case the equations are no longer triangular, and no analogue of
the formula for $C^w_{w,w}$ is known.  This added complexity was
handled in \cite{mihalcea:equivariant, mihalcea:equivariant*1} with
more refined recursive identities that reduce all (3 point, genus
zero) equivariant Gromov-Witten invariants to those associated to
multiplication with the identity class.  This leads to the surprising
realization that, even though the ring $\QH_T(X)$ may not be generated
by divisor classes, the structure of this ring is completely
determined by its Chevalley formula, and the same conclusion holds for
the equivariant cohomology ring $H^*_T(X;\Z)$.\footnote{After this
  paper was finished we learned from Ciocan-Fontanine that this can be
  explained using a localization argument.  See \Remark{generator} for
  details.}

In this section we show that the equivariant quantum $K$-theory ring
$\QK_T(X)$ is similarly determined by multiplication with divisor
classes, despite the fact that a formula for multiplication with
divisors is known only when $X$ is cominuscule.  We will derive this
result as a formal consequence of the same property of the equivariant
cohomology ring $H^*_T(X;\Z)$.  For this reason we start by
considering this case, where we reprove and slightly extend some
results from \cite{mihalcea:equivariant*1} in a form that will be
useful later.  This part of our analysis is also valid for the
equivariant cohomology of any Kac-Moody flag variety, see
\Remark{kac-moody}.

\subsection{Equivariant cohomology}\label{sec:ms.ht}

Let $\La = H^*_T(\pt;\Z)$ denote the equivariant cohomology of a point
and recall the notation from \Section{qk}.  The ring $H^*_T(X;\Z)$ is
a free $\La$-module with a basis of the Schubert classes $[X^u]$ for
$u \in W^P$.  The {\em equivariant Schubert structure constants\/} of
$X$ are the classes $C^w_{u,v} \in \La$, defined for $u,v,w \in W^P$
by the identity
\[
[X^u] \cdot [X^v] \ = \ \sum_w C^w_{u,v}\, [X^w]
\]
in $H^*_T(X;\Z)$.  The constant $C^w_{u,v}$ is non-zero only if
$u \leq w$ and $v \leq w$ in the Bruhat order of $W^P$.

For any character $\la : T \to \C^*$ we set
$c_T(\la) = c_1(\C_\la) \in \La$.  The ring $\La$ is generated by
these classes, and the classes $c_T(\be)$ for $\be \in \Delta$ are
algebraically independent.  A class in $\La$ is {\em non-negative\/}
if it can be written as a polynomial with non-negative integer
coefficients in the classes $c_T(\be)$ for $\be \in \Delta$.  Graham
has proved that all equivariant structure constants $C^w_{u,v}$ are
non-negative classes \cite{graham:positivity}.  For $\be \in \Delta$
we let $\omega_\be$ denote the corresponding fundamental weight.  Then
the Chevalley formula for $H^*_T(X;\Z)$
\cite{chevalley:decompositions, kostant.kumar:nil*1} states that for
$\be \in \Delta \ssm \Delta_P$ and $u,w \in W^P$ we have
\[
C^w_{s_\be,u} =
\begin{cases}
  c_T(\omega_\be - u.\omega_\be) & \text{if $w=u$,} \\
  (\omega_\be,\al^\vee) & \text{if $\ell(w)=\ell(u)+1$ and
    $\exists$ $\al\in \Phi^+$ with $w = u s_\al$,} \\
  0 & \text{otherwise.}
\end{cases}
\]

The following was proved in \cite{mihalcea:equivariant*1}.  We
summarize the proof for completeness.

\begin{lemma}\label{lemma:beta}
  Let $u,w \in W^P$ be minimal representatives.\smallskip

  \noin{\rm(a)} If $u \leq w$, then $C^w_{s_\be,w} - C^u_{s_\be,u}$ is
  a non-negative class in $\Lambda$ for every
  $\be \in \Delta \ssm \Delta_P$.\smallskip

  \noin{\rm(b)} If $u \neq w$, then there exists
  $\be \in \Delta \ssm \Delta_P$ such that
  $C^w_{s_\be,w} \neq C^u_{s_\be,u}$.\smallskip

  \noin{\rm(c)} If $u < w$ and $C^w_{s_\be,w} \neq C^u_{s_\be,u}$,
  then there exists $u' \in W^P$ such that $u < u' \leq w$,
  $\ell(u')=\ell(u)+1$, and $C^{u'}_{s_\be,u} \neq 0$.
\end{lemma}
\begin{proof}
  If $u \neq w$, then choose $\be \in \Delta \ssm \Delta_P$ such that
  $s_\be$ occurs in a reduced expression for $u^{-1}w$.  Then it
  follows from \cite[Thm.~1.12]{humphreys:reflection} that
  $u.\omega_\be \neq w.\omega_\be$.  This proves part (b).  Assume now
  that $u < w$.  Choose $v \in W^P$ such that $u < v \leq w$ and
  $\ell(v) = \ell(u)+1$, and choose $\al \in \Phi^+$ such that
  $v = u s_\al$.  Then $u.\al \in \Phi^+$.  For any
  $\be \in \Delta \ssm \Delta_P$ we have
  $u s_\al.\omega_\be = u.(\omega_\be - (\omega_\be,\al^\vee)\al) =
  u.\omega_\be - (\omega_\be,\al^\vee)u.\al$,
  so
  $C^v_{s_\be,v} - C^u_{s_\be,u} = c_T(u.\omega_\be - v.\omega_\be) =
  (\omega_\be,\al^\vee) c_T(u.\al)$
  is a non-negative class in $\La$.  Part (a) follows from this by
  induction on $\ell(w)-\ell(u)$.  Notice also that, when $u < v$ and
  $\ell(v) = \ell(u)+1$, we have $C^v_{s_\be,u}\neq 0$ if and only if
  $u.\omega_\be \neq v.\omega_\be$.  To prove part (c), assume that
  $u.\omega_\be = v.\omega_\be \neq w.\omega_\be$.  By induction on
  $\ell(w)-\ell(u)$ we can find $v' \in W^P$ such that $v < v' \leq w$
  and $C^{v'}_{s_\be,v} \neq 0$.  By part (b) we may choose
  $\ga \in \Delta \ssm \Delta_P$ with $C^v_{s_\ga,u} \neq 0$.  This
  implies that $[X^{v'}]$ occurs with non-zero coefficient in the
  product $[X^{s_\be}] \cdot [X^{s_\ga}] \cdot [X^u]$ in the ordinary
  cohomology ring $H^*(X;\Z)$.  Since this ring is associative, we
  deduce that $C^{u'}_{s_\be,u} \neq 0$ for some $u' \in W^P$ with
  $u < u' < v'$.
\end{proof}

Let $\La_0$ denote the field of fractions of $\La = H^*_T(\pt;\Z)$.
We will say that a class in $\La_0$ is {\em rationally positive\/} if
it can be written as a quotient of non-zero non-negative classes in
$\La$.  All rationally positive classes are non-zero, and all sums,
products, and quotients of rationally positive classes are again
rationally positive.  However, notice that some rationally positive
classes in $\La$ are not non-negative, for example we have
$x^2-x+1 = (x^3+1)/(x+1)$.

\begin{defn}\label{defn:ms_ht}
  Let $\{ D^w_{u,v} \}$ be any vector of classes $D^w_{u,v} \in \La_0$
  indexed by triples $(u,v,w) \in (W^P)^3$.  We will say that this
  vector satisfies the (generalized) {\em Molev-Sagan equations\/} for
  $H^*_T(X;\Z)$ if for all $u,v,w \in W^P$ and
  $\be \in \Delta \ssm \Delta_P$ we have
  \[
  (C^w_{s_\be,w} - C^u_{s_\be,u})\, D^w_{u,v} \ = \ \sum_{u' > u}
  C^{u'}_{s_\be,u} D^w_{u',v} - \sum_{w' < w} C^w_{s_\be,w'} D^{w'}_{u,v}
  \]
  and
  \[
  (C^w_{s_\be,w} - C^v_{s_\be,v})\, D^w_{u,v} \ = \ \sum_{v' > v}
  C^{v'}_{s_\be,v} D^w_{u,v'} - \sum_{w' < w} C^w_{s_\be,w'} D^{w'}_{u,v}
  \,,
  \]
  where the sums are over all elements $u'$, $v'$, or $w'$ in $W^P$
  with the indicated bounds.
\end{defn}

The associativity relations
$([X^{s_\be}] \cdot [X^u]) \cdot [X^v] = [X^{s_\be}] \cdot ([X^u]
\cdot [X^v])$
and
$[X^u] \cdot ([X^v] \cdot [X^{s_\be}]) = ([X^u] \cdot [X^v]) \cdot
[X^{s_\be}]$
of the ring $H^*_T(X;\Z)$ imply that the equivariant Schubert
structure constants $\{C^w_{u,v}\}$ of $X$ form a solution to the
Molev-Sagan equations.  It is known that the Molev-Sagan equations
together with known expressions \cite{kostant.kumar:nil*1} for the
special coefficients $C^w_{w,w}$ uniquely determine these structure
constants \cite{molev.sagan:littlewood-richardson,
  knutson.tao:puzzles, mihalcea:equivariant*1}.  It is interesting to
note that every solution $\{ D^w_{u,v} \}$ to the Molev-Sagan
equations must satisfy the commutativity relation
$D^w_{u,v} = D^w_{v,u}$, but we will not use this fact.

\begin{example}\label{example:molevsagan}
  Let $X = \bP^1$ and $W^P = \{1,s\}$, and set
  $a = C^s_{s,s} \in \La$.  Then all solutions $\{ D^w_{u,v} \}$ to
  the Molev-Sagan equations are given by
  \[
  D^s_{1,s} = D^s_{s,1} = \frac{1}{a} D^s_{s,s}
  \text{ \ and \ }
  D^s_{1,1} = \frac{1}{a^2} D^s_{s,s} - \frac{1}{a} D^1_{1,1}
  \text{ \ and \ }
  D^1_{1,s} = D^1_{s,1} = D^1_{s,s} = 0 \,,
  \]
  where $D^1_{1,1}$ and $D^s_{s,s}$ can be chosen freely in $\La_0$.
  Every solution to the Molev-Sagan equations defines an associative
  algebra $H = \Span_{\La_0}\{\sigma_1,\sigma_s\}$ with multiplication
  $\sigma_u \cdot \sigma_v = \sum_w D^w_{u,v} \sigma_w$.  The
  Chevalley formula holds in $H$ if and only if $D^s_{s,s} = a$.
\end{example}

\begin{prop}\label{prop:Auvwt}
  There exist classes $A^w_{u,v}(\tau) \in \La_0$, indexed by all
  quadruples $(u,v,w,\tau)$ in $(W^P)^4$, with the following
  properties:\smallskip

  \noin{\rm(a)} Every solution $\{ D^w_{u,v} \}$ to the Molev-Sagan
  equations for $H^*_T(X;\Z)$ satisfies
  \[
  D^w_{u,v} \ = \ \sum_{\tau\in W^P} A^w_{u,v}(\tau)
  D^\tau_{\tau,\tau} \,.
  \]

  \noin{\rm(b)} The class $A^w_{u,v}(\tau)$ is non-zero only if
  $u \leq w$, $v \leq w$, and $\tau \leq w$.\smallskip

  \noin{\rm(c)} If $u \leq w$ and $v \leq w$, then $A^w_{u,v}(w)$ is
  rationally positive.
\end{prop}
\begin{proof}
  Let $(u,v,w) \in (W^P)^3$ be any triple.  Assume by induction that
  $A^{w'}_{u',v'}(\tau)$ has been defined for all triples $(u',v',w')$
  for which $w'<w$, or $w'=w$ and $v'>v$, or $w'=w$ and $v'=v$ and
  $u'>u$, and for all $\tau \in W^P$.  Assume also that these classes
  satisfy properties (a), (b), and (c).  We define the classes
  $A^w_{u,v}(\tau)$ for $\tau \in W^P$ as follows.

  Assume first that $v \neq w$.  According to \Lemma{beta}(b) we may
  choose $\be \in \Delta \ssm \Delta_P$ such that
  $C^w_{s_\be,w} \neq C^v_{s_\be,v}$.  Using this choice we define
  \[
  A^w_{u,v}(\tau) \ = \ \frac{1}{C^w_{s_\be,w} - C^v_{s_\be,v}} \left(
    \sum_{v'>v} C^{v'}_{s_\be,v} A^w_{u,v'}(\tau) - \sum_{w'<w}
    C^w_{s_\be,w'} A^{w'}_{u,v}(\tau) \right) .
  \]
  If $u \neq v = w$, then choose $\be \in \Delta \ssm \Delta_P$ such
  that $C^w_{s_\be,w} \neq C^u_{s_\be,u}$ and define
  \[
  A^w_{u,w}(\tau) \ = \ \frac{1}{C^w_{s_\be,w} - C^u_{s_\be,u}}
  \sum_{u'>u} C^{u'}_{s_\be,u} A^w_{u',w}(\tau) \,.
  \]
  Finally, if $u=v=w$ then define
  \[
  A^w_{w,w}(\tau) = \delta_{w,\tau} \,.
  \]

  We must check that the classes $A^w_{u,v}(\tau)$ satisfy properties
  (a), (b), and (c).  These properties are clear if $u=v=w$.  In
  addition, property (b) follows immediately from the definition of
  $A^w_{u,v}(\tau)$ together with the induction hypothesis.  Assume
  that $v \neq w$.  Then property (a) holds because
  \[
  \begin{split}
    (C^w_{s_\be,w} - C^v_{s_\be,v})\, D^w_{u,v} \
    &= \
    \sum_{v'>v} C^{v'}_{s_\be,v} D^w_{u,v'} \ - \
    \sum_{w'<w} C^w_{s_\be,w'} D^{w'}_{u,v} \\
    &= \
    \sum_{v'>v} C^{v'}_{s_\be,v} \sum_\tau A^w_{u,v'}(\tau) D^\tau_{\tau,\tau}
    - \sum_{w'<w} C^w_{s_\be,w'} \sum_\tau A^{w'}_{u,v}(\tau)
    D^\tau_{\tau,\tau} \\
    &= \
    \sum_\tau \left( \sum_{v'>v} C^{v'}_{s_\be,v} A^w_{u,v'}(\tau)
      - \sum_{w'<w} C^w_{s_\be,w'} A^{w'}_{u,v}(\tau) \right) D^\tau_{\tau,\tau} \\
    &= \
    (C^w_{s_\be,w} - C^v_{s_\be,v})\, \sum_\tau A^w_{u,v}(\tau)
    D^\tau_{\tau,\tau} \,.
  \end{split}
  \]
  In addition, if $u\leq w$ and $v < w$, then it follows from property
  (b) that
  \[
  A^w_{u,v}(w) \ = \ \frac{1}{C^w_{s_\be,w}-C^v_{s_\be,v}}
  \sum_{v'>v} C^{v'}_{s_\be,v} A^w_{u,v'}(w) \,,
  \]
  and parts (a) and (c) of \Lemma{beta} together with property (c) of
  the induction hypothesis imply that this sum is rationally positive.
  A similar argument shows that properties (a) and (c) hold when
  $u \neq v = w$.
\end{proof}

\begin{cor}\label{cor:unique_sol_ht}
  Let $\{D^w_{u,v}\}$ be any solution to the Molev-Sagan equations for
  the ring $H^*_T(X;\Z)$ that also satisfies the additional condition
  \[
  \forall\, w \in W^P \ \exists\, u,v \in W^P : \, u \leq w \,\text{
    and }\, v \leq w \,\text{ and }\, D^w_{u,v} = C^w_{u,v} \,.
  \]
  Then we have $D^w_{u,v} = C^w_{u,v}$ for all $u,v,w \in W^P$.
\end{cor}
\begin{proof}
  In view of \Proposition{Auvwt} it is enough to show that
  $D^\tau_{\tau,\tau} = C^\tau_{\tau,\tau}$ for all $\tau \in W^P$.
  Let $w \in W^P$ and assume by induction that
  $D^\tau_{\tau,\tau} = C^\tau_{\tau,\tau}$ for all $\tau < w$.
  Choose $u,v \in W^P$ such that $u\leq w$, $v \leq w$, and
  $D^w_{u,v} = C^w_{u,v}$.  Then we have
  \[
  A^w_{u,v}(w) D^w_{w,w}
  = D^w_{u,v} - \sum_{\tau < w} A^w_{u,v}(\tau) D^\tau_{\tau,\tau}
  = C^w_{u,v} - \sum_{\tau < w} A^w_{u,v}(\tau) C^\tau_{\tau,\tau}
  = A^w_{u,v}(w) C^w_{w,w} \,.
  \]
  Since $A^w_{u,v}(w) \neq 0$ by \Proposition{Auvwt}(c), this implies
  that $D^w_{w,w} = C^w_{w,w}$.
\end{proof}

\begin{remark}
  The following instances of the additional condition of
  \Corollary{unique_sol_ht} have appeared: $D^w_{w,w} = C^w_{w,w}$ in
  \cite{molev.sagan:littlewood-richardson, knutson.tao:puzzles};
  $D^w_{w,1} = 1$ in \cite{mihalcea:equivariant*1}; and
  $D^w_{1,1} = \delta_{1,w}$ in \cite{mihalcea:equivariant*1}.
\end{remark}

\begin{cor}\label{cor:unique_alg_ht}
  Let $H$ be any associative $\La$-algebra that is also a free
  $\La$-module with a basis $\{ \sigma_u : u \in W^P \}$ indexed by
  $W^P$.  Assume that $H$ satisfies the (two-sided) Chevalley formula
  \[
  \sigma_{s_\be} \cdot \sigma_u \ = \ \sigma_u \cdot \sigma_{s_\be} \ = \
  \sum_{w \in W^P} C^{w}_{s_\be,u}\, \sigma_w
  \]
  for all $u \in W^P$ and $\be \in \Delta \ssm \Delta_P$, as well as
  the identity $\sigma_1 \cdot \sigma_1 = \sigma_1$.  Then the
  homomorphism of $\La$-modules $H \to H^*_T(X;\Z)$ defined by
  $\sigma_u \mapsto [X^u]$ is an isomorphism of rings.
\end{cor}
\begin{proof}
  Let $\{D^w_{u,v}\}$ be the structure constants of $H$ with respect
  to the basis $\{ \sigma_u \}$.  Then the relations
  $(\sigma_{s_\be} \cdot \sigma_u) \cdot \sigma_v = \sigma_{s_\be}
  \cdot (\sigma_u \cdot \sigma_v)$
  and
  $\sigma_u \cdot (\sigma_v \cdot \sigma_{s_\be}) = (\sigma_u \cdot
  \sigma_v) \cdot \sigma_{s_\be}$
  imply that $\{ D^w_{u,v} \}$ is a solution to the Molev-Sagan
  equations.  The identity $\sigma_1 \cdot \sigma_1 = \sigma_1$
  implies that $D^w_{1,1} = \delta_{w,1} = C^w_{1,1}$ for all
  $w \in W^P$.  It therefore follows from \Corollary{unique_sol_ht}
  that $D^w_{u,v} = C^w_{u,v}$ for all $u,v,w \in W^P$, as required.
\end{proof}

The following conjecture implies that the vector space of all
solutions to the Molev-Sagan equations has dimension equal to the
cardinality of $W^P$.

\begin{conj}
  The classes $A^w_{u,v}(\tau)$ of \Proposition{Auvwt} are uniquely
  determined by property {\rm(a)}.  Equivalently, for each fixed
  $\tau \in W^P$ the vector $\{ A^w_{u,v}(\tau) \}$ is the unique
  solution to the Molev-Sagan equations that also satisfies
  $A^w_{w,w}(\tau) = \delta_{w,\tau}$ for all $w \in W^P$.
\end{conj}

\begin{remark}\label{remark:kac-moody}
  The results proved in \Section{ms.ht} are true also when $X = G/P$
  is a homogeneous space defined by a Kac-Moody group $G$ and a
  parabolic subgroup $P$, with the same proofs.  The required
  Chevalley formula for the equivariant cohomology of such spaces is
  proved in Kumar's book \cite[Thm.~11.1.7]{kumar:kac-moody}.
\end{remark}

\begin{remark}\label{remark:generator}
  Ciocan-Fontanine brought to our attention that, under certain
  hypotheses, \cite[Lemma
  4.1.3]{ciocan-fontanine.kim.ea:abeliannonabelian} implies that the
  {\em localized\/} equivariant cohomology ring
  $H^*_T(X;\Z) \otimes_\La \La_0$ is generated by divisors.  We next
  give an alternative proof of this conclusion.  Let
  $D = \sum_{\be\in\Delta\ssm\Delta_P} [X^{s_\be}]$ be the sum of the
  Schubert divisors in $X$.
\end{remark}

\begin{lemma}\label{lemma:generator}
  The ring $H^*_T(X;\Z) \otimes_\La \La_0$ is generated by $D$ as a
  $\La_0$-algebra.
\end{lemma}
\begin{proof}
  Consider the $\La$-linear endomorphism
  $\mu_D : H^*_T(X;\Z) \to H^*_T(X;\Z)$ defined as multiplication by
  $D$.  Since multiplication by $D$ can be written in the form
  $D \cdot [X^u] = D_u [X^u] + \sum_{w>u} c(u;w)\,[X^w]$, where
  $c(u;w) \in \La$ and $D_u$ is the restriction of $D$ to the
  $T$-fixed point $u.P$, it follows that $\mu_D$ is represented by an
  upper-triangular matrix.  The diagonal entries are given by
  $D_u = c_T(\omega - u.\omega)$, where
  $\omega = \sum_{\be \in \Delta\ssm\Delta_P} \omega_\be$ (see
  \cite{kostant.kumar:t-equivariant} or e.g.\
  \cite[Thm.~8.1]{buch.mihalcea:curve}).  Since we have
  $v.\omega = \omega$ if and only if $v \in W_P$ \cite[V,
  \S4.6]{bourbaki:elements*51}, we deduce that $\mu_D$ has distinct
  diagonal entries.  This implies that the minimal polynomial of
  $\mu_D$ is equal to the characteristic polynomial, so the elements
  $1, D, D^2, \dots, D^{\#W_P-1}$ are linearly independent over
  $\La_0$ and span $H^*_T(X;\Z)\otimes_\La \La_0$.
\end{proof}

In relation to \cite{ciocan-fontanine.kim.ea:abeliannonabelian},
consider the $T$-equivariant embedding $X \subset \bP(V)$, where
$V = H^0(X,\cO_X(D))$.  In order to apply
\cite[Lemma~4.1.3]{ciocan-fontanine.kim.ea:abeliannonabelian} to this
situation, one must show that the restriction map
$H^*(\bP(V)^T;\C) \to H^*(X^T;\C) = \C^{\# W^P}$ is surjective.  In
fact, this follows from \Lemma{generator} together with the
commutative diagram used in the proof of
\cite[Lemma~4.1.3]{ciocan-fontanine.kim.ea:abeliannonabelian}.
However, it is not in general true that $\bP(V)$ has isolated
$T$-fixed points.  For example, this fails when $X$ is an {\em adjoint
  flag variety\/} in the sense that $V = \Lie(G)$ is the adjoint
representation of $G$, see e.g.\ \cite{chaput.perrin:quantum}.

\subsection{Equivariant quantum $K$-theory}

We finally show that the ring $\QK_T(X)$ is uniquely determined by
multiplication with the divisor classes and the identity element.  Our
argument is based on linear algebra and could also be applied to the
equivariant $K$-theory ring $K^T(X)$ or the equivariant quantum
cohomology ring $\QH_T(X)$.  For $u,v,w \in W^P$ we define the power
series
\[
N^w_{u,v} \ = \ \sum_{d \geq 0} N^{w,d}_{u,v}\, q^d
\]
in the ring
$\Gq = \Gamma\llbracket q_\be : \be \in \Delta \ssm
\Delta_P\rrbracket$.  The product in $\QK_T(X)$ is given by
\[
\cO^u \star \cO^v \ = \ \sum_{w\in W^P} N^w_{u,v}\, \cO^w \,.
\]
Let $\Gq_0$ be the field of fractions of $\Gq$.

\begin{defn}\label{defn:ms_qkt}
  Let $\{D^w_{u,v}\}$ be a vector of classes $D^w_{u,v} \in \Gq_0$
  indexed by triples $(u,v,w) \in (W^P)^3$.  We say that this vector
  satisfies the generalized {\em Molev-Sagan equations\/} for
  $\QK_T(X)$ if for all $u,v,w \in W^P$ and
  $\be \in \Delta \ssm \Delta_P$ we have
  \[
  \sum_{u' \in W^P} N^{u'}_{s_\be,u} D^w_{u',v} \ = \
  \sum_{w' \in W^P} N^w_{s_\be,w'} D^{w'}_{u,v} \ = \
  \sum_{v' \in W^P} N^{v'}_{s_\be,v} D^w_{u,v'} \,.
  \]
\end{defn}

The equalities
$(\cO^{s_\be} \star \cO^u) \star \cO^v = \cO^{s_\be} \star (\cO^u
\star \cO^v) = \cO^u \star (\cO^{s_\be} \star \cO^v)$
imply that the structure constants $\{N^w_{u,v}\}$ form a solution to
the Molev-Sagan equations for $\QK_T(X)$.  Notice that if each
structure constant $N^w_{s_\be,u}$ in \Definition{ms_qkt} is replaced
with $C^w_{s_\be,u}$, then the Molev-Sagan equations for $\QK_T(X)$
turn into the Molev-Sagan equations for $H^*_T(X;\Z)$.  One noteworthy
difference between the two cases is that the structure constants
$N^w_{u,v}$ of $\QK_T(X)$ may be non-zero when $u \not\leq w$ and
$v \not\leq w$.  For this reason the Molev-Sagan equations for
$\QK_T(X)$ are not triangular, which makes it time-consuming to solve
them.  Another obstruction to solving the equations is that the
constants $N^w_{s_\be,u}$ are known only when $X$ is a cominuscule
variety.  Here we will only discuss how the ring $\QK_T(X)$ is
determined by the Molev-Sagan equations.  A systematic method for
solving the Molev-Sagan equations for the equivariant quantum
cohomology ring $\QH_T(X)$ can be found in
\cite{mihalcea:equivariant*1}.

\begin{prop}\label{prop:unique_sol_qkt}
  Let $\{D^w_{u,v}\}$ be any solution to the Molev-Sagan equations for
  $\QK_T(X)$ that also satisfies the additional condition
  \[
  \forall\, w \in W^P \ \exists\, u,v \in W^P : \, u \leq w \,\text{
    and }\, v \leq w \,\text{ and }\, D^w_{u,v} = N^w_{u,v} \,.
  \]
  Then we have $D^w_{u,v} = N^w_{u,v}$ for all $u,v,w \in W^P$.
\end{prop}

Before we prove \Proposition{unique_sol_qkt}, we first recall some
facts about the equivariant $K$-theory ring $K^T(X)$.  Let
$\ch : \Gamma \to \wh\La := \prod_{k=0}^\infty H^{2k}_T(X;\Q)$ denote
the {\em equivariant Chern character\/} from
\cite[\S3.1]{edidin.graham:riemann-roch}.  This injective ring
homomorphism is defined by $\ch([\C_\la]) = \exp(c_T(\la))$.  For
$\sigma \in K^T(X)$ we may write $\ch(\sigma) = \sum \ch_m(\sigma)$,
where $\ch_m(\sigma) \in H^{2m}_T(X;\Q)$ is the term of degree $m$.
Set $\ch_m(\sigma)=0$ for $m<0$.  We will say that $\sigma$ has {\em
  degree at least\/} $k$ if $\ch_m(\sigma)=0$ for all $m<k$.  If
$\sigma \neq 0$, then the {\em leading term\/} of $\sigma$ is the
first non-zero term $\ch_m(\sigma)$ (with $m$ minimal).  For
$u,v,w \in W^P$ we set $d(u,v,w) = \ell(u)+\ell(v)-\ell(w)$.  The
structure constant $N^{w,0}_{u,v}$ of $K^T(X)$ has degree at least
$d(u,v,w)$, and we have $\ch_{d(u,v,w)}(N^{w,0}_{u,v}) = C^w_{u,v}$.
This can be proved geometrically using the equivariant Riemann-Roch
formula \cite[Cor.~3.1]{edidin.graham:riemann-roch} (see
\cite[\S4.1]{buch.mihalcea:quantum}), or combinatorially using the
Chevalley formula for $K^T(G/B)$ \cite{lenart.postnikov:affine}.

\begin{proof}[Proof of \Proposition{unique_sol_qkt}]
  According to \Corollary{unique_sol_ht}, the vector $\{C^w_{u,v}\}$
  of structure constants of $H^*_T(X;\Z)$ is the unique solution to
  the equations of \Definition{ms_ht} together with a choice of
  $k = \# W^P$ additional equations of the form $D^w_{u,v}=C^w_{u,v}$.
  Since this is a system of linear equations, we can choose a subset
  of $k^3$ of the equations that are sufficient to determine the
  vector $\{C^w_{u,v}\}$.  Consider the corresponding set of equations
  for $\QK_T(X)$.  More precisely, if the chosen equations for
  $\{C^w_{u,v}\}$ include the first (resp.\ second) equation of
  \Definition{ms_ht} for some $u,v,w \in W^P$ and
  $\be \in \Delta \ssm \Delta_P$, then we use the first (resp.\
  second) equality of \Definition{ms_qkt} given by the same elements
  $u,v,w,\be$.  For each equation of the form $D^w_{u,v}=C^w_{u,v}$ we
  similarly use the corresponding equation $D^w_{u,v}=N^w_{u,v}$.  We
  claim that the vector $\{N^w_{u,v}\}$ is the unique solution to
  these equations for $\QK_T(X)$.  These equations can be stated in
  matrix form as $AD=b$, where $A$ is a $k^3 \times k^3$ matrix and
  $b$ is a vector of length $k^3$, both with entries from $\Gq$, and
  $D = \{D^w_{u,v}\}$ is the vector of indeterminates.  Since
  $\{N^w_{u,v}\}$ satisfies the Molev-Sagan equations for $\QK_T(X)$,
  it suffices to show that the determinant of $A$ is not zero.  For
  this it is enough to show that $\det(\ov{A}) \neq 0$, where $\ov{A}$
  is obtained from $A$ by substituting zero for all quantum parameters
  $q_\be$.  We will show that the leading term of $\det(\ov{A})$ is
  the determinant of the corresponding system of linear equations for
  $H_T(X;\Z)$.

  Number the rows and columns of $A$ from 1 to $k^3$, and define
  integers $d_i$ and $d'_j$ for $1 \leq i,j \leq k^3$ as follows.  If
  the $i$-th row of $A$ represents an equation coming from
  \Definition{ms_qkt} for a triple $(u,v,w) \in (W^P)^3$, then set
  $d_i = d(u,v,w)+1$.  If the $i$-th row of $A$ represents (the left
  hand side of) an equation $D^w_{u,v} = N^w_{u,v}$, then set
  $d_i = d(u,v,w)$.  Finally, if the $j$-th column of $A$ corresponds
  to the indeterminate $D^w_{u,v}$, then set $d'_j = d(u,v,w)$.  It
  follows by inspection of the Molev-Sagan equations that each entry
  $\ov{A}_{ij}$ of $\ov{A}$ has degree at least $d_i-d'_j$.  This
  implies that $\det(\ov{A})$ has degree at least
  $\sum_i d_i - \sum_j d'_j$.  Furthermore, if we replace each entry
  $\ov{A}_{ij}$ of $\ov{A}$ with $\ch_{d_i-d'_j}(\ov{A}_{ij})$, then
  the result is the matrix of coefficients of the original $k^3$
  equations for the structure constants of $H^*_T(X;\Z)$.  Since this
  system of linear equations is known to have a unique solution, we
  deduce that the term of degree $\sum d_i - \sum d'_j$ in
  $\ch(\det(\ov{A}))$ is indeed non-zero.  This completes the proof.
\end{proof}

\begin{cor}\label{cor:unique_alg_qkt}
  Let $H$ be any associative algebra over $\Gq$ that is also a free
  $\Gq$-module with a basis $\{\sigma_u : u \in W^P\}$ indexed by
  $W^P$.  Assume that $H$ satisfies
  \[
  \sigma_{s_\be} \star \sigma_u \ = \ \sigma_u \star \sigma_{s_\be} \ = \
  \sum_{w\in W^P} N^w_{s_\be,u}\, \sigma_w
  \]
  for all $u \in W^P$ and $\be \in \Delta \ssm \Delta_P$, as well as
  the identity $\sigma_1 \star \sigma_1 = \sigma_1$.  Then the
  homomorphism of $\Gq$-modules $H \to \QK_T(X)$ defined by
  $\sigma_u \mapsto \cO^u$ is an isomorphism of rings.
\end{cor}

\begin{remark}\label{remark:gk}
  Gorbounov and Korff have used ideas from integrable systems to
  construct an algebra called $qh^*(\Gr(m,n))$, together with a basis
  for this algebra that corresponds to the Schubert classes of the
  Grassmannian $\Gr(m,n)$ \cite{gorbounov.korff:quantum}\footnote{The
    parameter $\be$ used in \cite{gorbounov.korff:quantum} must be
    replaced with $-1$ in order to match our notation.}.  It follows
  from \cite[Cor.~3.18]{gorbounov.korff:quantum} that this algebra
  satisfies a Chevalley formula that agrees with our
  \Theorem{qktchev}.  We therefore deduce from
  \Corollary{unique_alg_qkt} that the algebra $qh^*(\Gr(m,n))$ has the
  same Schubert structure constants as the equivariant quantum
  $K$-theory ring $\QK_T(\Gr(m,n))$.  This proves that the two
  algebras are isomorphic, which is Conjecture~1.2 in
  \cite{gorbounov.korff:quantum}.  In particular, the presentation
  (Thm.~5.16) and Giambelli formula (Cor.~5.13) proved in
  \cite{gorbounov.korff:quantum} are valid for the equivariant quantum
  $K$-theory of Grassmannians.  We note that the relation of
  $qh^*(\Gr(m,n))$ with the equivariant quantum cohomology ring
  $\QH_T(\Gr(m,n))$ and the (non-equivariant) quantum $K$-theory ring
  $\QK(\Gr(m,n))$ was established in \cite{gorbounov.korff:quantum} by
  using the structure theorems from \cite{mihalcea:equivariant*1,
    buch.mihalcea:quantum}.  An alternative geometric proof of the
  presentation of $\QK(\Gr(m,n))$ has recently been given by Gonzalez
  and Woodward by writing the Grassmannian as a GIT quotient of a
  space of matrices \cite{gonzalez.woodward:quantum}.
\end{remark}


\providecommand{\bysame}{\leavevmode\hbox to3em{\hrulefill}\thinspace}
\providecommand{\MR}{\relax\ifhmode\unskip\space\fi MR }
\providecommand{\MRhref}[2]{%
  \href{http://www.ams.org/mathscinet-getitem?mr=#1}{#2}
}
\providecommand{\href}[2]{#2}

\end{document}